\documentclass[%
 aip,
 amsmath,amssymb,
 reprint,%
]{revtex4-1}

\usepackage{graphicx}
\usepackage{caption}
\usepackage{subcaption}
\usepackage[export]{adjustbox}
\usepackage{dcolumn}
\usepackage{bm}

\usepackage[utf8]{inputenc}
\usepackage[T1]{fontenc}
\usepackage{mathptmx}
\usepackage{etoolbox}
\usepackage{amsmath,bm}
\usepackage{xcolor}
\usepackage{amsmath}
\usepackage{amsthm}
\usepackage{mathtools}
\usepackage{commath}
\usepackage{enumitem,kantlipsum}
\DeclarePairedDelimiterX{\inp}[2]{\langle}{\rangle}{#1, #2}

\usepackage[capitalise]{cleveref}
\theoremstyle{plain}

\crefname{assumption}{Assump.}{Assumption}
\newtheorem{definition}{Definition}
\crefname{definition}{Definition}{Definition}

\newtheorem{theorem}{Theorem}
\crefname{theorem}{Theorem}{Theorems}

\crefname{conjecture}{Conjecture}{Conjectures}
\newtheorem{remark}{Remark}
\crefname{remark}{Remark}{Remarks}

\crefname{lemma}{Lemma}{Lemmas}
\newtheorem{proposition}{Proposition}
\crefname{proposition}{Proposition}{Propositions}
\newtheorem{corollary}{Corollary}
\crefname{corollary}{Corollary}{Corollaries}

\crefname{example}{Example}{Examples}

\crefname{claim}{Theorem}{Claim}

\usepackage{acronym}

\makeatletter
\def\@email#1#2{%
 \endgroup
 \patchcmd{\titleblock@produce}
  {\frontmatter@RRAPformat}
  {\frontmatter@RRAPformat{\produce@RRAP{*#1\href{mailto:#2}{#2}}}\frontmatter@RRAPformat}
  {}{}
}%

\makeatother

\begin{document}

\preprint{AIP/123-QED}

\title[A Minimal Set of Koopman Eigenfunctions]{A Minimal Set of Koopman Eigenfunctions\\Analysis and Numerics}
\author{I. Cohen} \email{ido.coh@gmail.com.}
 \affiliation{ 
Department of Mathematics, Technion
}%
\author{E. Appleboim}\email{eliap@ee.technion.ac.il.}
\affiliation{ 
Electrical and Computer Engineering Department, Technion
}%


\date{\today}

\begin{abstract}
Research on Koopman operator theory has focused on three key areas for several decades: the mathematical structure of the Koopman eigenfunction space, the basis of this space, and the ability to represent nonlinear dynamics as linear. This study provides a thorough and comprehensive framework for these topics, including theoretical, analytical, and numerical approaches. A novel mathematical structure is introduced, which outlines permissible actions on the infinite set of Koopman Eigenfunction, under which this set is closed. Notions of generating and independent sets of Koopman eigenfunctions are defined. In addition, notions of a minimal generating set, and a maximal independent set are defined and are shown to be equivalent. This structure defines conditions for independence within the set of Koopman eigenfunctions. This independent set can be interpreted as a new coordinate system in which the dynamical system is linear. The theory also highlights the equivalence of a minimal set, flowbox representation, and conservation laws. Finally, the presented theory is supported by numerical experiments.
\end{abstract}

\maketitle

\begin{quotation}
Given an $N$ dimensional dynamical system, the conclusions from this work are as follows.
\begin{description}
    \item[Cardinality] There are \emph{only} $N$ independent solutions for Koopman PDE that generate the space of these solutions termed as a \emph{minimal set}.  
    \item[System Reconstruction] From them, the governing laws are revealed as well as the conservation laws.
    \item[Equivalency] Flowbox representation, conservation laws, and a minimal set are found to be equivalent to each other.
    \item[Precise and Global Linearity] A minimal set is a coordinate system in which this dynamic is linear.
\end{description}
\end{quotation}

\section{Introduction}
The Koopman spectrum is a commonly used tool for the analysis of dynamical systems. Treating Koopman Eigenfunction space as an infinite dimensional space yields various techniques to represent the system as a linear one with truncated dimensionality \cite{schmid2010dynamic,brunton2016sparse,kaiser2018discovering,servadio2023koopman,schmid2022dynamic}. Naturally, these methods occasionally result in a redundant and overly large decomposition \cite{mezic2005spectral,williams2015data}, which can sometimes be inaccurate \cite{cohen2021modes,cohen_gilboa_2023}. Thus, the challenge of extracting meaningful information about the dynamics from samples remains open \cite{avila2020data}. In this study, we present analytical and numerical frameworks to identify the minimal set of Koopman Eigenfunctions required to perfectly recover the dynamic from samples. 

Regrettably, the mathematical framework of the Koopman spectrum has received little attention, resulting in a limited understanding of how to efficiently extract a representation based on the underlying geometry of this space from samples \cite{bollt2021geometric}. Due to this lack of knowledge, unsophisticated and exhaustive algorithms \cite{brunton2022data,williams2015data,li2017extended} have been employed, but their assumptions have led to intrinsic flaws in dynamical representation and prediction, such as in highly nonlinear time-variant systems \cite{turjeman2022underlying}, homogeneous flows \cite{cohen2021modes}, or even linear systems with non-zero inputs \cite{lu2021extended}. This study aims to bridge this knowledge gap by presenting a comprehensive theory on the mathematical structure of the set of Koopman eigenfunctions.

\subsection{Main contributions}
The contribution of this work is discussed in three phases, theoretical, analytical, and numerical parts. 
\paragraph{\bf{Theoretic Part}} 
A novel viewpoint on the set of Koopmann eigenfunctions is suggested. This viewpoint enables the definition of actions on this set under which the set is closed. It also enables the definitions of an independent set and of a generating set. Using these definitions it is possible to prove that a maximal independent set (appropriately defined) is a minimal generating set (defined properly).

\paragraph{\bf{Analytic Part}} The notion of a minimal generating set induces a new representation of the dynamic. It is shown that under some specific conditions, a minimal set can serve as a coordinate system with respect to which the dynamic becomes linear. In addition, it is shown that there is a transformation from a flowboxed coordinate system to a minimal generating set.

\paragraph{\bf{Numeric Part}} Given a vector field of a system, a numeric method to find a minimal generating set is presented. Based on geometric constraints of independence set and constant velocity restrictions, a functional is formulated where a minimal set is its minimizer. This functional is given by a loss function fed into a Neural Network training process.  

\paragraph{\bf{Experiments}}
The theoretic, analytic, and numeric part are backed-up with experiments. From the numeric results, theory restrictions, analytic solutions, and numeric solutions are perfectly compatible. In addition, we thoroughly discuss the importance of density and diversity of samples in system recovery.

\section{Setup and motivation}
\subsection{Notations and definitions}
\begin{description}
    \item[Dynamic]  Let us consider the following nonlinear dynamical system, defined in a domain $\mathcal{D}$ in $\mathbb{R}^n$
\begin{equation}\label{eq:dynamics}
    \dot{\bm{x}}=P(\bm{x}), \,\, t\in I=[0,T]
\end{equation}
where $\bm{x}\in \mathbb{R}^N$, the operator $\,\dot{\,\,}\,$ denotes the time derivative, and $P:\mathbb{R}^N\to \mathbb{R}^N$. All along this work it is assumed that $P$ is $C^1$ and therefore $\bm{x}(t)\in C^2$.

\item[Orbit of an initial point] Given an initial condition, $\bm{x}(t=0)= \bm{x}_0$, the unique solution of \eqref{eq:dynamics} can be seen as a curve in $\mathbb{R}^N$. This trajectory is denoted by $\mathcal{X}(\bm{x}_0)$, and termed as the \emph{orbit of $\bm{x}_0$}.

\item[Equilibrium] An equilibrium point, denoted by $\bm{x}^*\in \mathcal{D}$, is a stationary point of Eq.  \eqref{eq:dynamics}, i.e. a point at which
\begin{equation}
    P(\bm{x}^*)=\bm{0}
\end{equation}
where $\bm{0}$ is the $N$ dimensional zero vector.

\item[Measurement] A measurement is a function from $\mathcal{D}$ to $\mathbb{C}$. 

\item[Koopman Operator] The Koopman operator $K_\tau$ acts on the infinite dimensional vector space of measurements and admits the following.  Let $g(\bm{x})$ be a measurement then
\begin{equation}\label{eq:koopDef}\tag{\bf{KO}}
    K_\tau(g(\bm{x}(s)))=g(\bm{x}(s+\tau)),\quad s,s+\tau\in I,
\end{equation}
where $\tau>0$. This operator is linear \cite{koopman1931hamiltonian,mezic2005spectral}.
\item[Koopman \acs{PDE}] Let $\Phi(\bm{x})$ be some differentiable function from $\mathcal{D}$ to $\mathbb{C}$. Then $\Phi$ is a solution of the Koopman \ac{PDE} if it satisfies the following, everywhere in $\mathcal{D}$,
\begin{equation}\label{eq:KEFPDE}
    \nabla \Phi (\bm{x})^T P(\bm{x}) = \lambda \Phi(\bm{x}), \quad \forall \bm{x}\in \mathcal{D}. 
\end{equation}
where $\nabla$ denotes the gradient of $\Phi$ with respect to the state vector $\bm{x}$. In particular, it is assumed that $\Phi$ is $C^1$ as a function of $\bm{x}$. 

\item[Koopman Manifold] Under the assumptions of differentiable dynamics the graph of a solution of Eq. \eqref{eq:KEFPDE} can be regarded as a manifold in $\mathbb{R}^{N}\times \mathbb{C}$. We term this manifold as \emph{Koopman Manifold}. 

\item[\acl{KEF}] Assuming the initial condition $\bm{x}_0$, a measurement $\varphi(\bm{x})$, satisfying the following relation along the orbit $\mathcal{X}(\bm{x}_0)$
\begin{equation}\label{eq:KEFdiff}
    \dfrac{d\varphi (\bm{x})}{dt}=\lambda \varphi(\bm{x}), \quad \forall \bm{x}\in \mathcal{X}(\bm{x}_0)
\end{equation}
for some value $\lambda\in \mathbb{C}$, is a \acf{KEF}.
According to the existence and uniqueness theorem, $\varphi(\bm{x})$ is given along the orbit of $x_0$ by:
\begin{equation}\label{eq:KEFform}
    \varphi(\bm{x}(t)) = \varphi(\bm{x}_0)e^{\lambda t}
\end{equation}
Existence of Koopman Eigenfunction is thoroughly discussed in \cite{cohen_gilboa_2023}. 
Koopman eigen functions and solutions of the Koopman \ac{PDE} are related by the following way:
Let $\Phi(\bm{x})$ be a solution of Koopman PDE, and let $\mathcal{X}(\bm{x}_0)$ be an orbit of the dynamic initiated at $\bm{x}_0$, and contained in $\mathcal{D}$. Then, one can derive a Koopman eigenfunction $\varphi(\bm{x})$ from $\Phi(\bm{x})$ simply by
\begin{equation}
    \varphi(\bm{x}) = \Phi(\bm{x}), \quad \forall \bm{x}\in \mathcal{X}(\bm{x}_0).
\end{equation}
That is, $\varphi$ is the restriction of $\Phi$ along the orbit of $x_0$. Note however, that a \ac{KEF} need not necessarily be a restriction of a solution of the Koopman PDE. An example of a Koopman eigenfunction that is not such a restriction is given in \cref{appsec:KEF_KM}. 
\end{description}

\subsection{Mathematical structure induced on the set of \acp{KEF}}
\paragraph{\bf{Multiplicity of Koopman Eigenfunctions}}
The set of Koopman eigenfunctions is closed under a variety of operations. For instance, given a \ac{KEF} $\varphi$ that is everywhere non-zero, with some eigenvalue $\lambda \ne 0$, one can generate a \ac{KEF} from $\varphi$ with any other eigenvalue, since $(\varphi)^\beta$ is also a \ac{KEF} for all $\beta$. In addition, if $\varphi_1(\bm{x})$ and $\varphi_2(\bm{x})$ are \acp{KEF} then $\varphi_3(\bm{x})=\varphi_1(\bm{x})\varphi_2(\bm{x})$ is also one. Hence, it is closed under this operation. Furthermore, point-wise multiplication of \acp{KEF} is associative, and the constant function $1$ serves as a unit element with respect to this operation (it is an eigenfunction of $\lambda = 0$), and for every eigenfunction, $e^{\lambda t}$ the eigenfunction $e^{-\lambda t}$ is the inverse element. As a result the set of \acp{KEF}s has an abelian group structure with respect to point-wise multiplication. In fact, this set of is closed also under other types of point-wise operations, yet no full characterization of these operations had been made.
\paragraph{\bf{Algebraic-differential structure}}
In the sequel, a definition of an algebraic-differential structure of the set of \ac{KEF}s is suggested. The newly suggested structure calls for a new point of view on the set of Koopman eigenfunctions that is different from the common viewpoint which studies the spectrum of the Koopman operator. 
Thus, the focus is moved from the Koopman spectrum to the ability to generate the set of solutions of the Koopman PDE in some appropriate sense that will be defined in Section $5$. Special interest will be taken in small generating sets, minimal such sets in particular. The main merit of this result is that a minimal generating set can serve as a new coordinate system in which the dynamic is linear, and the transformation from the coordinate system to the linear one is invertible.

\section{Set Up: Unit manifolds and Time Mappings }

\begin{definition}[Unit velocity measurement]\label{def:unitVelocity}
A unit velocity measurement is a smooth function $M:\mathcal{D}\subset \mathbb{R}^N \to \mathbb{C}$  satisfying the following PDE \footnote{This function is denoted with $M$ as a shortage for $\mu o v \acute\alpha \delta \alpha$ unit in Greek.},
\begin{equation}\label{eq:unitPDE}
    \nabla M(\bm{x})^T P(\bm{x})=1,\quad \forall \bm{x}\in \mathcal{D}.
\end{equation}
\end{definition}

\begin{definition}[Unit manifold]\label{def:unitManifold}
A unit manifold is the graph of a unit velocity measurement. 
\end{definition}
Given a solution of Eq. \eqref{eq:dynamics} $\bm{x}(t)$, for some initial condition $\bm{x}_0$, with the orbit $\mathcal{X}(\bm{x}_0)\subset \mathcal{D}$, the time derivative of the measurement $M(\bm{x})$ is 
\begin{equation}
    \frac{d}{dt}M(\bm{x})=\nabla M(\bm{x})^T P(\bm{x})=1, 
\end{equation}
which is the source of its name.

\paragraph{\bf{From solutions of the Koopman PDE to unit velocity measurements and back.}} Let $\Phi$ be a function satisfying Eq. \eqref{eq:KEFPDE} for an eigenvalue $\lambda\ne 0$, so that $\Phi(\bm{x})\ne 0,\, \forall \bm{x}\in \mathcal{D}$, then 
\begin{equation}\label{eq:inducedUnitManifold}
    M(\bm{x})=\frac{1}{\lambda}\ln(\Phi(\bm{x}))
\end{equation}
is a unit velocity measurement. Namely, $M(\bm{x})$ admits Eq. \eqref{eq:unitPDE}
\begin{equation}
    \nabla M(\bm{x})^T P(\bm{x})=\frac{\nabla \Phi(\bm{x})^T P(\bm{x})}{\lambda\Phi(\bm{x})}= \frac{\lambda\Phi(\bm{x})}{\lambda\Phi(\bm{x})}=1.
\end{equation}
In order for this to be well defined, a specific branch should be chosen for the logarithmic function, as a function of the complex variable $\Phi(\bm{x})$.
Under the conditions that $\lambda \ne 0$ and that $\Phi(\bm{x})\ne 0, \, \forall \bm{x}\in \mathcal{D}$, a solution for Eq. \eqref{eq:KEFPDE} entails a solution for Eq. \eqref{eq:unitPDE} and vice versa. 

Following Eq. \eqref{eq:inducedUnitManifold}, the gradients of $M(\bm{x})$ and $\Phi(\bm{x})$ satisfy the following:
\begin{equation}\label{eq:UnitKEF_Mani_Rel}
    \nabla \Phi(\bm{x}) = \lambda \Phi(\bm{x}) \nabla M(\bm{x}), \quad \forall \bm{x}\in \mathcal{D}.
\end{equation}

\begin{definition}[Time mappings]\label{def:timeMapping}
Let $\bm{x}(t)$ be a solution of the dynamical system initiated at $\bm{x}(t=0) = \bm{x}_0$ for $t\in I$. Given that the trajectory $\bm{x}(t)$ is regular and has no self intersections, the time mapping is defined as the inverse function from the state in the orbit $\mathcal{X}(\bm{x}_0)$ to the corresponding time point in the interval $I$, more formally, $\mu:\mathcal{X}(\bm{x}_0) \to I$ such that
\begin{equation}\label{eq:timeMapping}
    \mu(\bm{x}(t)) = t, \quad \forall \bm{x}\in \mathcal{X}(\bm{x}_0).
\end{equation}
\end{definition}

\paragraph{\bf{From unit velocity measurments to time mappings and back.}} If the subset $\mathcal{X}(\bm{x}_0)\subset\mathcal{D}$, 
then a time mapping is related to the unit velocity measurement $M(\bm{x})$ in $\mathcal{X}(\bm{x}_0)$ as
\begin{equation}
    \mu(\bm{x};\bm{x}_0) = M(\bm{x})-M(\bm{x}_0), \quad \forall \bm{x}\in \mathcal{X}(\bm{x}_0).
\end{equation}
This expression is a time mapping since the time derivative of $\mu(\bm{x};\bm{x}_0)$ is one and $\mu(\bm{x}_0;\bm{x}_0)=0$. Following the existence and uniqueness theorem $\mu(\bm{x};\bm{x}_0)$ admit Eq. \eqref{eq:timeMapping}.

\paragraph{\bf{From Koopman eigenfunctions to time mappings and back.}} Let $\varphi$ be a Koopman eigenfunction defined on $\mathcal{X}(\bm{x}_0)$. Assuming $\lambda\ne 0$, and choosing a branch for the complex logarithmic function, a time mapping, $\mu_{\varphi}(\bm{x};\bm{x}_0)$, induced from $\varphi$ is given by 
\begin{equation}\label{eq:inducedTM}
    \mu_{\varphi}(\bm{x};\bm{x}_0) = \frac{1}{\lambda}\left[\ln(\varphi(\bm{x}))-\ln(\varphi(\bm{x}_0))\right], \quad \forall \bm{x}\in \mathcal{X}(\bm{x}_0).
\end{equation}
This function is a time mapping since the time derivative of $\mu_{\varphi}(\bm{x};\bm{x}_0)$ is one and $\mu_{\varphi}(\bm{x}_0;\bm{x}_0)=0$.

\subsection{Koopman Conjugation}
Since projecting solutions of the Koopman PDE onto unit velocity measurements is not $1:1$ we will define the following conjugacy relation between different solutions of Koopman PDE.
\begin{definition}[Koopman PDE Conjugation]\label{def:KEFConjugationIdo}
We will say that two solutions of Koopman PDE. $\Phi_1(x), \Phi_2(x)$ associated with the eigenvalues $\lambda_1,\lambda_2$ are {\emph conjugate} if both are mapped via the complex logarithm function to the same unit velocity measurement $M(\bm{x})$ up to a constant for all $\bm{x}\in\mathcal{D}$
\begin{equation}\label{eq:conjugation}
\begin{split}
    \frac{1}{\lambda_1}\ln(\Phi_1(\bm{x})) &= \frac{1}{\lambda_2}\ln(\Phi_2(\bm{x})) +c
\end{split}.
\end{equation}
\end{definition} 
It is evident that the conjugacy relation just defined is an equivalence relation on the set of all solutions of the Koopman PDE. In addition, each two representatives of the same conjugacy class differ from each other by power and a coefficient.

Restricting Eq. \eqref{eq:conjugation} to an orbit yields the following conjugacy relation.
\begin{remark}[Koopman Eigenfunction Conjugation]
The same is true for the map from the set of Koopman eigenfunctions to the set of time mappings.
\end{remark}

To recap, clear relations are shown between unit velocity measurement and solutions of Koopman PDE and between \ac{KEF}s and time mappings, summarized in Fig. \ref{fig:relationSummary}. Going the other way around, above each unit velocity measurement obtained by Eq. \eqref{eq:inducedUnitManifold} there exists a fiber composed of infinitely many solutions of the Koopman PDE that are in the same conjugacy class. Thus, the given unit velocity measurement can be lifted to one of these solutions arbitrarily chosen.

\begin{figure}[phtb!]
    \centering 
    \includegraphics[trim=0 150 0 100, clip,width=0.5\textwidth,valign = t]{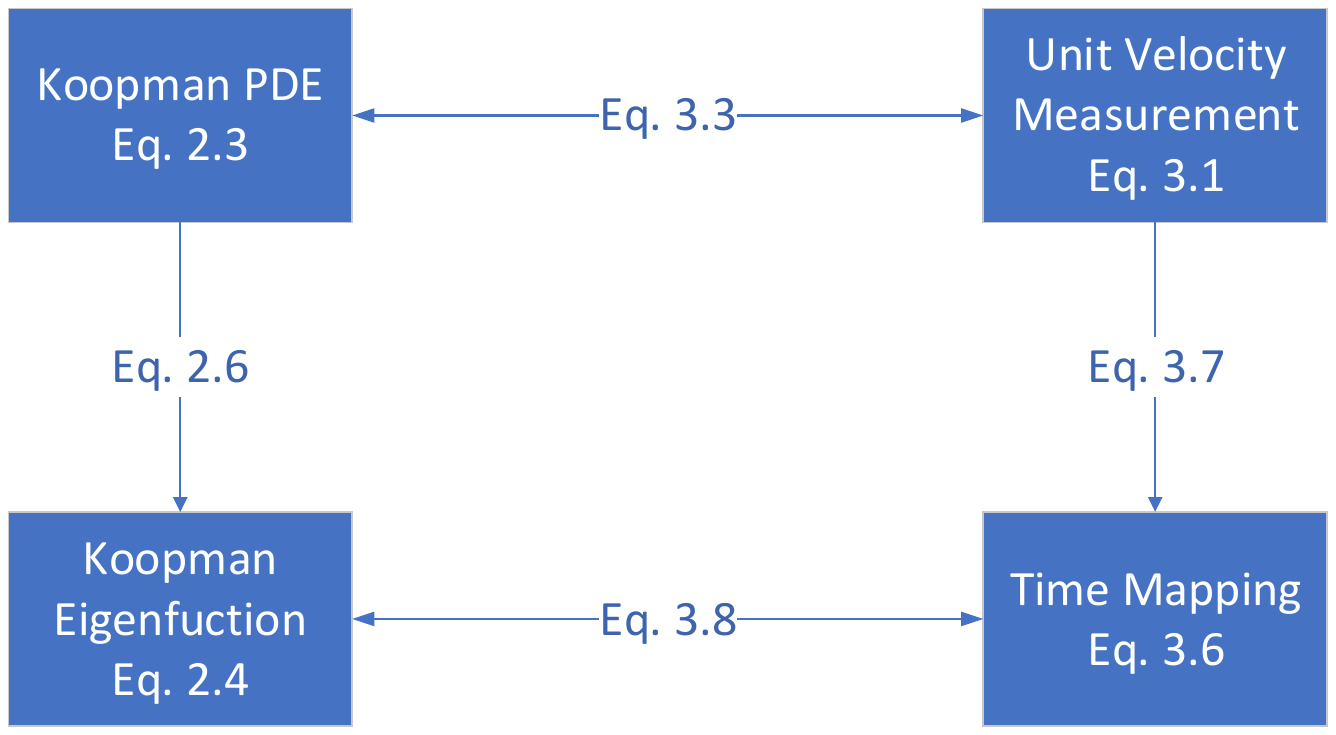}
    \caption{Relation Summary -- Koopman PDE, Koopman Eigenfunction, Unit velocity measurement, and time mapping.}
    \label{fig:relationSummary}
\end{figure}

\section{Algebraic-Differential structure on the set of \ac{KEF}s}
In this section a new mathematical structure on the set of unit velocity measurements and on the set of \acp{KEF}s is defined. This structure will further induce a structure on the set of solutions of the Koopman PDE.
The first step is to rigorously define the set of operations under which the infinite sets of \acp{KEF} and solutions of Koopman PDE are closed. In order to do so, time mappings and unit velocity measurements, defined above, will serve as a bypass. Describing the set of allowed operations under which the set of time mappings is closed is actually straightforward and is given by the following.

\begin{definition}[Admissible shift of unit velocity measurements]
Let $\{M_i\}_{i=1}^K$ be any finite set of unit velocity measurements from $\mathcal{D}\subset \mathbb{R}^N$ to $\mathcal{M}_i\subseteq \mathbb{C}$. Let $f$ be an analytic function from $\mathcal{M}_1\times \cdots\times \mathcal{M}_K$ to $\mathbb{C}$, and let its partial derivatives be denoted by $\{\partial_i f\}_{i=1}^K$. The function $f$ is an admissible shift on the set $\{M_i\}_{i=1}^K$ if the following relation holds:
\begin{equation}\label{eq:legalActionUM}
    \nabla f^T(\bm{M}(\bm{x})) P(\bm{x}) =1, \quad \forall \bm{x}\in \mathcal{D},
\end{equation}
where $\bm{M}(\bm{x})=\begin{bmatrix}
    M_1(\bm{x})&\cdots&M_K(\bm{x})
\end{bmatrix}$.
One can reformulate this condition by using the chain rule
\begin{equation}
    \nabla f P(\bm{x})=\sum_{i=1}^K \partial_i f \underbrace{\nabla M_i(\bm{x})^TP(\bm{x})}_{=1, \,\, Eq. \eqref{eq:unitPDE}} = \sum_{i=1}^K \partial_i f = 1.
\end{equation}
\end{definition}

\begin{definition}[Admissible time shift]\label{def:legalActions}
Let $\{\mu_i\}_{i=1}^K$ be any finite set of time mappings from an orbit $\mathcal{X}(\bm{x}_0)$ to the time interval $I$. Let $f$ be a differentiable function from $I^K$ to $I$, and $\{\partial_i f\}_{i=1}^K$ denote its partial derivatives. The function $f$ is an admissible time shift (acting on the set $\{\mu_i\}_{i=1}^K$ and the result is a time mapping) if it admits the following relation
\begin{equation}\label{eq:legalActionTM}
    \frac{df}{dt}=\sum_{i=1}^K\partial_i f = 1,
\end{equation}
at every point on the main diagonal of $I^K$, i.e. along the line $\mu_i=\mu_j, \, 1\leq i,j\leq K$, $\mu_i\in I,\, 1\leq i \leq K$, and $f(\bm{0})=0$ where $\bm{0}$ is a $K$ dimensional zero vector.
\end{definition}

One can see an admissible shift as any manipulation of time mappings that keeps the "physical" unit as "time". For example, let $\varphi$ and $\vartheta$ be \acp{KEF} where the corresponding eigenvalue $1$, and let  $\mu_{\varphi}$ and $\mu_{\vartheta}$ be the corresponding time mappings (Eq. \eqref{eq:inducedTM}). One can generate different time mappings as $\tilde{\mu}=\frac{\mu_{\varphi}+\mu_{\vartheta}}{2}$ or  $\hat{\mu} = \sqrt{\mu_{\varphi} \cdot \mu_{\vartheta}}$. Clearly, $\tilde{\mu}$ and $\hat{\mu}$ are time mappings since there time derivatives are $1$ and $\tilde{g}(\bm{x}_0)=\hat{g}(\bm{x}_0)=0$. In addition, these time mappings are induced from the \acp{KEF} $\varpi=\sqrt{\varphi\cdot \vartheta}$ and $\varrho = \exp\{\sqrt{\left(\ln\varphi-\ln\varphi_0\right)\left(\ln\vartheta-\ln\vartheta_0\right)}\}$, respectively.

\section{Minimal set}
Based on the algebraic-differential structure defined above, we prove there is a finite set of solutions of Koopman PDE that generates the whole space of these solutions. As a result, it is possible also to define a minimal generating set. Using the linear algebra simile, this finite set can be described as a "basis" of this space. Due to the projections of solutions of Koopman PDE onto unit velocity measurements, and of Koopman eigenfunctions on time mappings, and due to the conjugacy relations previously defined it suffices to prove the existence of a ``basis'' for unit velocity measurements. The proof for time mappings is identical, and then for solutions of the Koopman PDE, and Koopman eigenfunction one can use pullback arguments to obtain a ``basis'' for the set of conjugacy classes.

\subsection{Minimal Set }
\subsubsection{Unit Velocity Measurements -- Generating, Independence, Minimal Set}

\begin{definition}[Generating set]\label{def:generatingSetTM}
Let $\{M_i\}_{i=1}^K$ be a set of unit velocity measurements. This set is called a generating set for the entire set of unit velocity measurements if any unit velocity measurement $M$ can be presented as some admissible shift (as in Definition \ref{def:legalActions}) acting on this set.
\end{definition}

\begin{definition}[Generated set]\label{def:generatedSetTM}
Let $\{M_i\}_{i=1}^K$ be a set of unit velocity measurements then its generated set ,$\mathcal{G}(\{M_i\}_{i=1}^K)$, is the set of all the unit velocity measurements spanned by $\{M_i\}_{i=1}^K$ under the actions in Definition \ref{def:legalActions}.
\end{definition}

\begin{definition}[Geometric independence]\label{def:independentTM2}  
Let $\{M_i\}_{i=1}^K$ be a set of unit velocity measurements. This set is independent if the set of vectors $\{\nabla M_i\}_{i=1}^K$ is linearly independent for all $\bm{x}\in \mathcal{D}$.
\end{definition}

\begin{definition}[Algebraic--Deferential independence]\label{def:independentTM1} 
Let $\{M_i\}_{i=1}^K$ be a set of unit velocity measurements. This set is non-degenerated if $M_j\notin \mathcal{G}(\{M_i\}_{i=1,i\ne j}^M)$ for all $j=1,\cdots, K$.
\end{definition}

\begin{proposition} \label{prop:equivalence} 
Definitions \ref{def:independentTM2} and \ref{def:independentTM1} are equivalent.
\end{proposition}
\begin{proof} Let $\{M_i\}_{i=1}^K$ be a degenerated set, i.e. there is $M_j$ in this set, and an admissible shift $f$, such that $M_j = f(\{M_i\}_{i=1,i\ne j}^K)$. Using the chain rule, one can formulate the gradient of $M_j$ as
\begin{equation}
    \nabla M_j(\bm{x}) =  \sum_{i=1,i\ne j}^K \partial_i f \nabla M_i(\bm{x}).
\end{equation}
Hence, the set $\{M_i\}_{i=1}^K$ is a dependent set according to Definition \ref{def:independentTM2}.

Next, let $\{M_i\}_{i=1}^K$ be a dependent set according to Definition \ref{def:independentTM2}, i.e. there is a unit velocity measurement $M_j(\bm{x})$ for which the following holds:
\begin{equation}
    \nabla M_j(\bm{x})=\sum_{i=1,i\ne j}^{K} a_i(\bm{x}) \nabla M_i(\bm{x}).
\end{equation}
According to the existence and uniqueness theorem, there exist a function, $f$ such that
\begin{equation}\label{eq:fDependent}
    \nabla_M f = \begin{bmatrix}a_1(\bm{x})&\cdots&a_{j-1}(\bm{x})&a_{j+1}(\bm{x})&\cdots&a_K(\bm{x})\end{bmatrix}^T,
\end{equation}
with the initial condition 
\begin{equation}
    M_j(\bm{x}_0)=f(\{M_i\}_{i=1,i\ne j}^K(\bm{x}_0)).
\end{equation}
Therefore, the set $\{M_i\}_{i=1}^K$ is a degenerated set according to Definition \ref{def:independentTM1}.
Note, the function $f$ from Eq. \eqref{eq:fDependent} is an admissible shift in term of Definition \ref{def:legalActions} since $M_j$ is a unit velocity measurement.
\end{proof}

 \begin{proposition}[Maximal cardinality of a set of independent unit velocity measurements]\label{prop:Maximal cardinality} 
 Any $N+1$ unit velocity measurements are dependent.
\end{proposition}
\begin{proof}
    Let $\{M_i\}_{i=1}^{N+1}$ be a set of $N+1$ unit velocity measurements. Conversely, assume this set is independent. According to \cref{def:independentTM2} the set $\{\nabla M_i\}_{i=1}^{N+1}$ is linearly independent which is impossible since $\nabla M_i$ is an $N$ dimensional vector. Therefore, the cardinality of the largest independent unit velocity measurement set is up to $N$.
\end{proof}

\begin{definition}[Maximal independent set]\label{def:maximalIndependentTM} Let $\mathcal{G}$ be an infinite set of unit velocity measurements. Let $G = \{M_i(\bm{x})\}_{i=1}^K$ be an independent set of unit velocity measurements, where $G\subset\mathcal{G}$. It is called a maximal independent set if any set $\hat{G}$ which strictly contains $G$ is dependent.
\end{definition}

\begin{definition}[Minimal generating set]\label{def:minGeneratingSet} Let $\mathcal{G}$ be an infinite set of unit velocity measurements. A generating set of $\mathcal{G}$ is a minimal if any strict subset of it does not generate $\mathcal{G}$.
\end{definition}

\begin{theorem}[Minimal generating set and maximal independent set]\label{theo:minimal} Let $\mathcal{G}$ be the set of all unit velocity measurements of the dynamic $P$. Any maximal independent set of $\mathcal{G}$ is also a minimal generating set of $\mathcal{G}$ and vice versa.
\end{theorem}
\begin{proof}
    Let $G$ be a maximal independent set of $\mathcal{G}$. Let $\mathcal{H}$ be the generated set of $G$. We would like to show that $\mathcal{H}=\mathcal{G}$. Obviously $\mathcal{H} \subseteq \mathcal{G}$. Suppose there is some $M$ in $\mathcal{G}$ and not in $\mathcal{H}$ then $\{M,G\}$ is independent, according to \cref{def:independentTM1} and \cref{prop:equivalence}. This contradicts the fact that $G$ is a maximal independent set. Therefore, $G$ is a generating set of $\mathcal{G}$. Now we are left to prove that all elements in $G$ are essential to generate $\mathcal{G}$, however, 
    it is clear since $G$ is independent. 

    Now, let $G$ be a minimal generating set of $\mathcal{G}$. According to either Definition \ref{def:independentTM2} or Definition \ref{def:independentTM1} $G$ must be independent otherwise it is not minimal. In addition, for any function $M$ in $\mathcal{G}$ the set $\{M, G\}$ is dependent because $G$ itself already generates $\mathcal{G}$. Hence, $G$ is a maximal indendent set.
\end{proof}

\begin{corollary}[Finite Cardinality of Generating Set of Koopman PDE Solution]
    The cardinality of a generating set of unit velocity measurements is finite. Then, also the cardinality of the set of conjugacy classes of solutions of the Koopman PDE is finite. Thus, if we limit our discussion to the Koopman Eigenfunction set which are restrictions of the solutions of Koopman PDE, the set also has a generating set with finite cardinality. 
\end{corollary}

By \cref{prop:Maximal cardinality} the cardinality of a minimal set is finite and does not exceed the dimensionality of the system. In the rest of this paper, it is assumed that the cardinality is maximal. The discussion about the conditions under which the cardinality is maximal exceeds the frame of this paper, however, there is a reference to that in the numerical results. Such a minimal set induces a new coordinate system for which the velocity equals $1$ in each coordinate. Upon the condition of independence, the new dynamic is called \emph{canonical split dynamic} if the gradients are all orthogonal to each other, defined as follows.

\begin{definition}[Canonical split dynamic]\label{def:canonicalSplitDynamic} Let $\{M_i\}_{i=1}^N$ be a minimal set where $\{\nabla M_i\}_{i=1}^N$ is an orthogonal set for all $\bm{x}\in \mathcal{D}$. A canonical dynamic splitting is the dynamic represented in the coordinate system $\hat{y}_i=M_i(\bm{x})$. Then, the dynamic can be reformulated as
    \begin{equation}\label{eq:canonicalDynamicSplitting}
    \begin{split}
        \dot{\hat{y}}_1&=1\\
        \quad\vdots\\
        \dot{\hat{y}}_N&=1
    \end{split}.
\end{equation}
Obviously, $\det\left(J(\hat{\bm{y}})\right)\ne 0$ for all $\bm{x}\in\mathcal{D}$.
\end{definition}

The word "canonical" stands for the independence between the coordinates, for the unit velocity for each coordinate, and for the orthogonality of the gradients. One can get a split dynamic only by independence between the coordinates and under the condition that $\det\left(J(\hat{\bm{y}})\right)\ne 0$ for all $\bm{x}\in\mathcal{D}$. Note, that the canonical split dynamic is not unique. Unfortunately, the discussion about the conditions under which this dynamic decomposition exists exceeds the frame of this paper.

Recall that the flowbox theorem states that given a Lipschitz vector field, there is an invertible transformation from a  neighborhood of a point, that is far a way from a singularity of the system, to a coordinate system for which the vector field is trivial, i.e. unit velocity in one coordinate and zero in the rest  \cite{calcaterraboldt2008flowbox}. As a result from \cref{def:canonicalSplitDynamic}, a flowboxed coordinate system is a rotation and rescaling of a coordinate system induced from a minimal set. Thus, a minimal set leads to a flowboxed coordinate system and vice versa.

\section{Minimal Set - Numeric Part} In this part, we translate the insight from the analytic part to build a neural network (NN for short) to find a minimal set of unit velocity measurements and indirectly a minimal set of conjugacy classes of solutions of the Koopman PDE. Feeding the NN with the vector field in $\mathcal{D}$, $P(\bm{x})$, yields a minimal set of unit velocity measurements, $\{M_i(\bm{x})\}_{i=1}^N$ in $\mathcal{D}$. The numeral part is based on \cref{def:canonicalSplitDynamic}. Thus, the loss function naturally emerges from this definition, given by
\begin{equation}\label{eq:functional}
 \mathcal{L}(\bm{x}) = \underbrace{\sum_{i=1}^N\left(\inp{\nabla M_i(\bm{x})}{P(\bm{x})}-1\right)^2}_{A}+\underbrace{\sum_{i=1}^{N-1}\sum_{j=i+1}^N\inp{\nabla M_i}{\nabla M_j}^2}_{B}.
\end{equation}
The summand $A$ guarantees unit velocity for each measurement. The $B$ term guarantees the orthogonality.
While the convergence of $A$ is necessary to find $N$ unit velocity measurements, the convergence of $B$ is not, but only to assure different manifolds. Thus, $B$ can be factorized by a constant $0<\lambda<1$ according to the system in question. In the next, section results from NN are presented using the loss function in \cref{eq:functional}.

\section{Results}
In this section, the theory discussed above is demonstrated by applying it to 2D linear and nonlinear dynamical systems. For every system, the analytic and numeric solutions are presented. The functional \cref{eq:functional} is used as the loss function and gets the following form
\begin{equation}\label{eq:functional2D}
    \begin{split}
        \mathcal{L}&=\left(\inp{\nabla M_1(\bm{x})}{P(\bm{x})}-1\right)^2+\\
        &\quad\left(\inp{\nabla M_2(\bm{x})}{P(\bm{x})}-1\right)^2+\\
        &\quad \inp{\nabla M_1}{\nabla M_2}^2.
    \end{split}.
\end{equation}
The general setting is as follows. In each example, the vector field is known, and the analytic and numeric results are presented. In addition, we demonstrate failure in finding a general minimal set when the solutions of the Koopman PDE are zero.   
\subsection{Linear systems}
Let us consider the following linear systems of the form
\begin{equation}\label{eq:2DlinearSys}
    \begin{bmatrix}
        \dot{x}_1\\\dot{x}_2
    \end{bmatrix}=A\begin{bmatrix}
        x_1\\x_2
    \end{bmatrix}
\end{equation}
where $A$ is a $2\times 2$ matrix gets the values of $A_r,\, A_C$, and $A_I$
\begin{equation*}
    A_R = \frac{1}{2}\begin{bmatrix}
        11&-5\\
        -5&11
    \end{bmatrix}, \quad A_C = \frac{1}{10}\begin{bmatrix}
        -4&1\\-4&-5
        \end{bmatrix}, \quad A_I = \begin{bmatrix}
        0&1\\-1&0
    \end{bmatrix}.
\end{equation*}
These systems have either real, complex, or imaginary eigenvalues, respectively.

\subsubsection{Real eigenvalues}

\paragraph{\bf{Analytic part}}
The eigenpairs of $A_R$ are $\{8,\frac{1}{\sqrt{2}}\begin{bmatrix}
    1\\-1
\end{bmatrix}\}$, $\{3,\frac{1}{\sqrt{2}}\begin{bmatrix}
    1\\1
\end{bmatrix}\}$, and the solution is 
\begin{equation}
    \begin{bmatrix}
        x_1\\x_2
    \end{bmatrix}= \frac{a_1}{\sqrt{2}}\begin{bmatrix}
        1\\-1
    \end{bmatrix}e^{8t}+\frac{a_2}{\sqrt{2}}\begin{bmatrix}
        1\\1
    \end{bmatrix}e^{3t}.
\end{equation}
The solutions of Koopman PDE are $\Phi_1 = \dfrac{1}{\sqrt{2}}(x_1+x_2)$ and $\Phi_2 = \dfrac{1}{\sqrt{2}}(x_1-x_2)$ where $\nabla \Phi_1$ and $\nabla \Phi_2$ are orthogonal. As long as $\Phi_1,\Phi_2\ne 0$ one can formulate a minimal set of unit velocity measurements are
\begin{equation}\label{eq:linRealEVSplitTM}
    M_1=\frac{1}{3}\ln{\abs{\Phi_1}},\quad 
        M_2=\frac{1}{8}\ln{\abs{\Phi_2}},
\end{equation}
which follows the canonical split dynamic definition. The flowboxed coordinates are
\begin{equation}\label{eq:linRealEVSplitFB}
    z_1=\dfrac{M_1+M_2}{2},\quad 
        z_2=\dfrac{M_1-M_2}{2}
\end{equation}

\paragraph{\bf{Numeric part}} In \cref{fig:2DLinReal}, the domain $\mathcal{D} =[4,6]\times [1,3]$ is depicted. The first row is the analytic solution, Eq. \eqref{eq:linRealEVSplitFB}, and the second row is the numeric one. The left column is the vector field in the system coordinates (blue), level sets of the flowbox coordinates $z_0$ (black) and $z_1$ (red). The right column is the vector field in the flowbox coordinates.

\begin{figure*}[phtb!]
    \centering
    \captionsetup[subfigure]{justification=centering}
    \begin{subfigure}[t]{0.48\textwidth} 
\includegraphics[width=0.85\textwidth,valign = t]{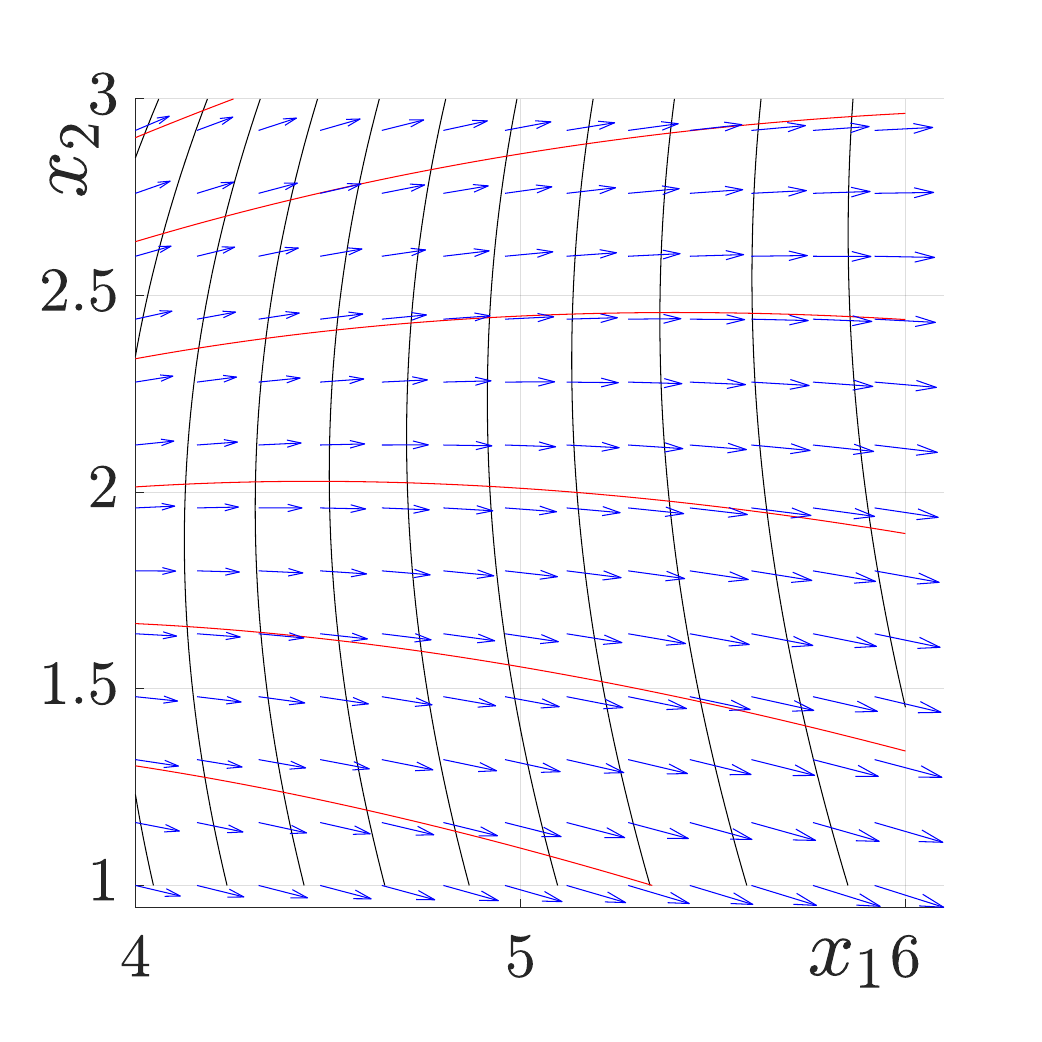}
        \label{subfig:linearSysRealEVFBcon}
    \end{subfigure}
    \begin{subfigure}[t]{0.48\textwidth} 
\includegraphics[width=0.85\textwidth,valign = t]{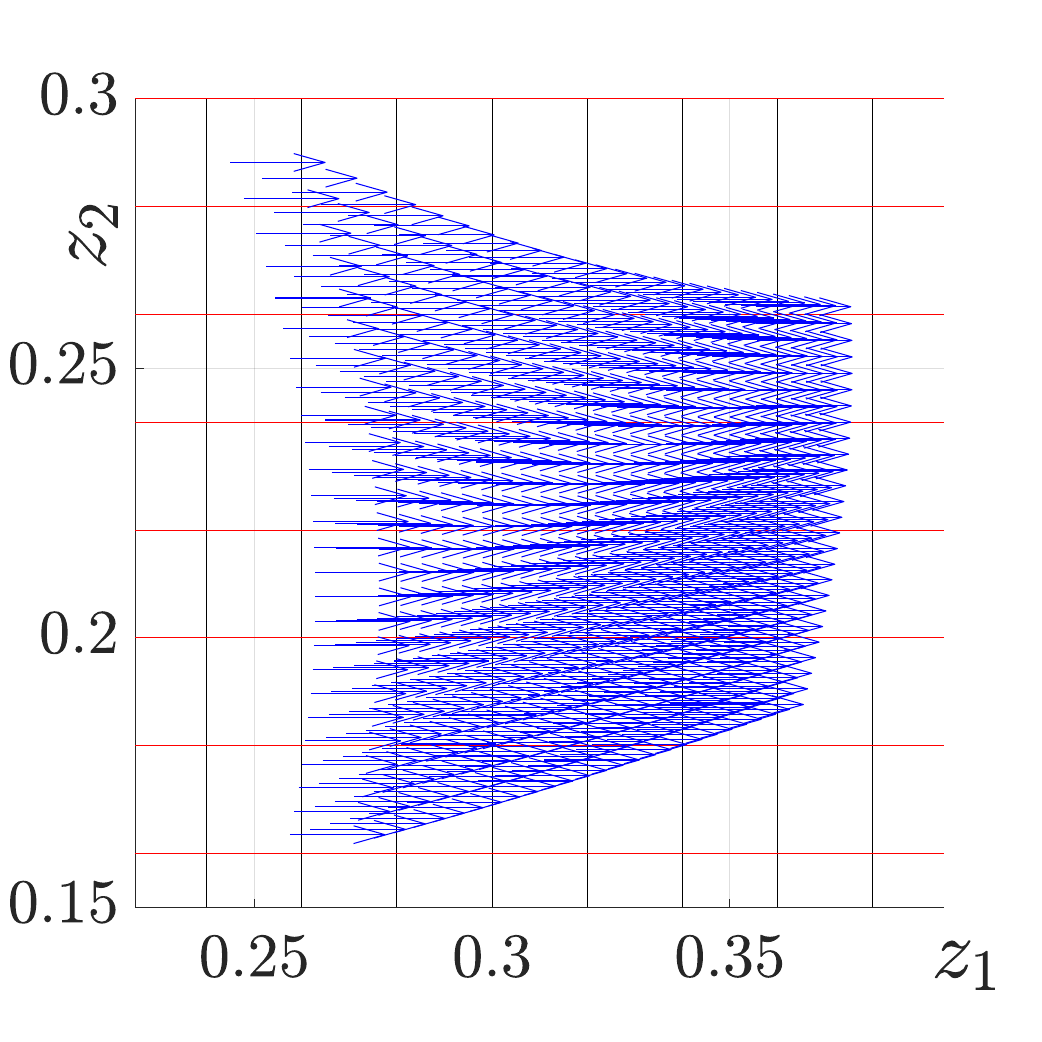}
        \label{subfig:linearSysRealEVFB}
    \end{subfigure}
        \label{fig:linearSysRealEVAnalytic}
        \caption*{Analytic solution of the system $A_R$ the flowbox coordinates via a minimal set}    
    \begin{subfigure}[t]{0.48\textwidth} 
\includegraphics[width=0.85\textwidth,valign = t]{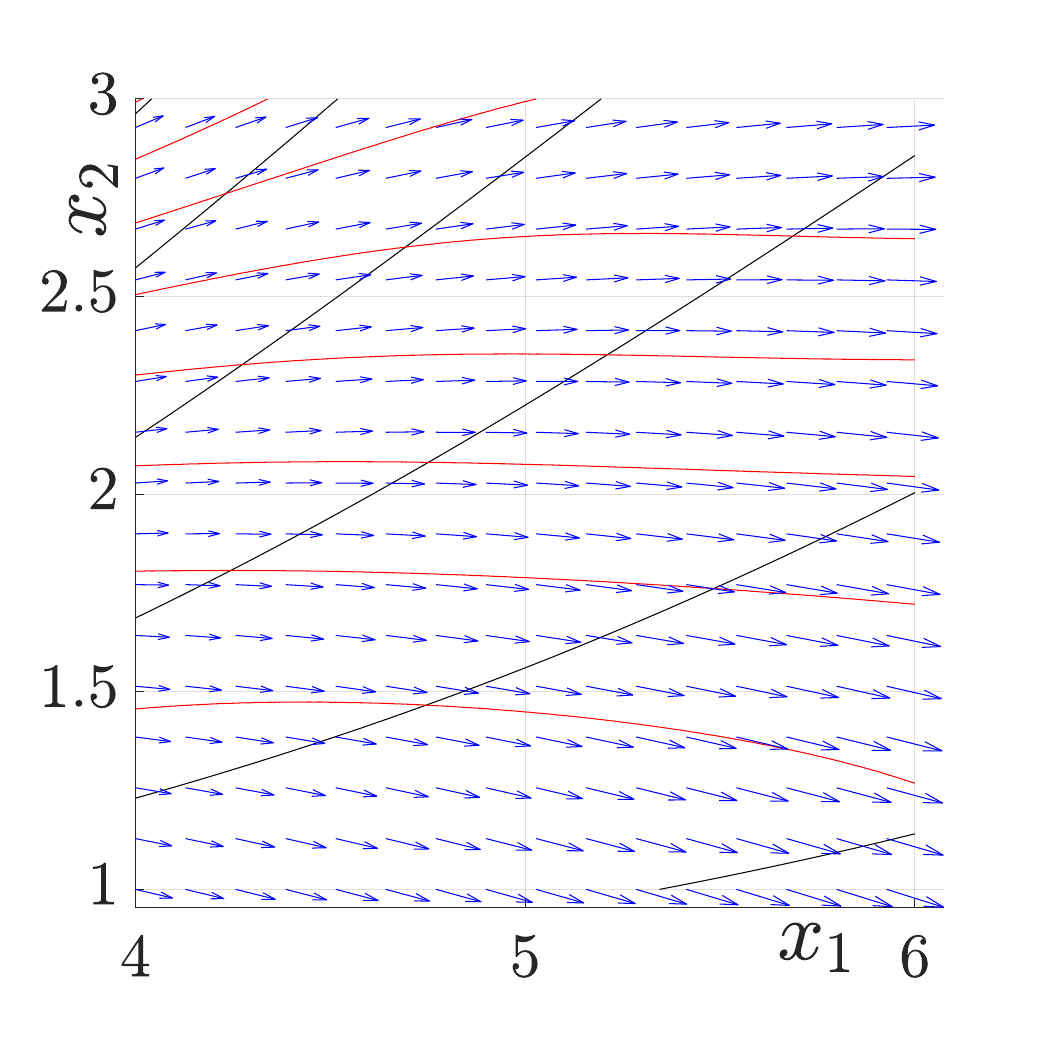}
        \label{subfig:linearSysRealEVFBconNumwithoutFoliationCleanRe}
    \end{subfigure}
    \begin{subfigure}[t]{0.48\textwidth} 
\includegraphics[width=0.85\textwidth,valign = t]{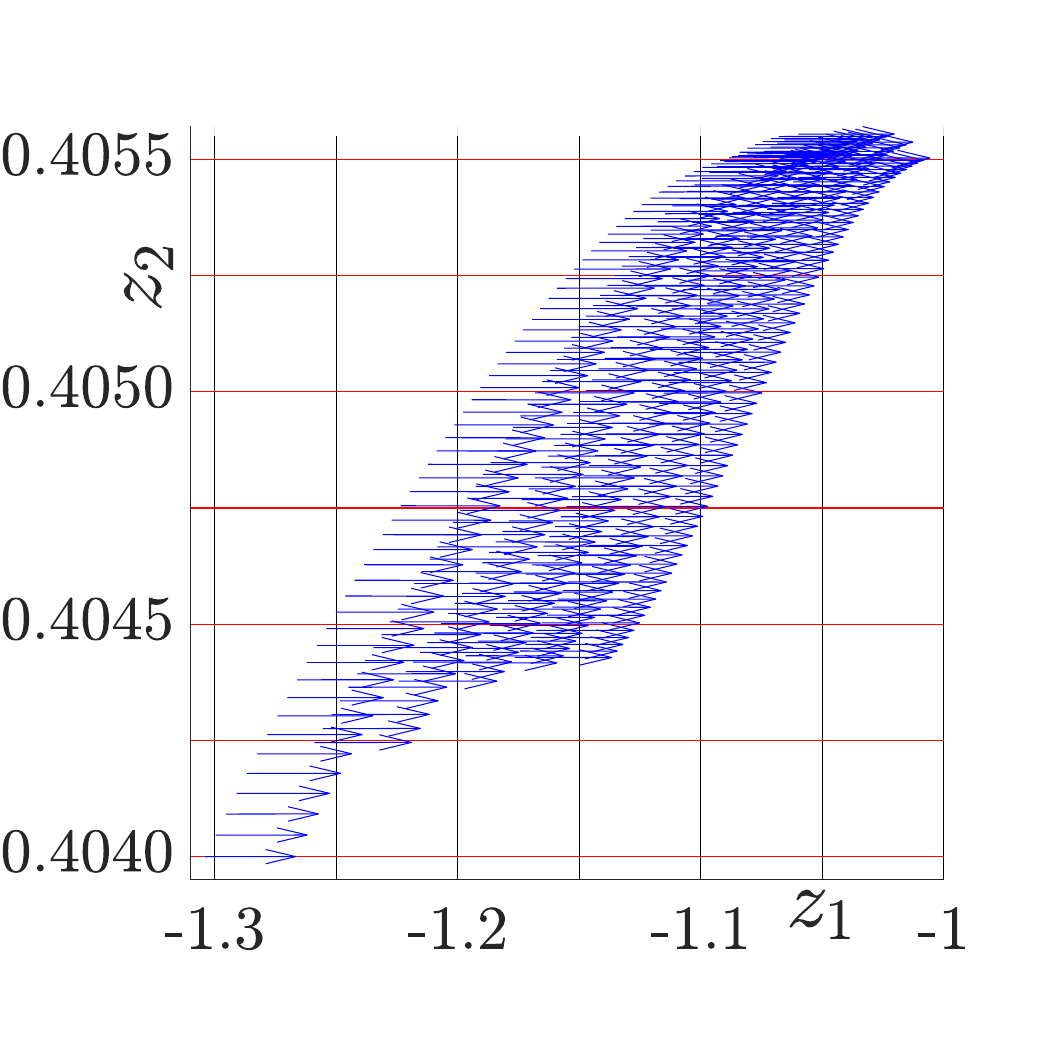}
        \label{subfig:linearSysRealEVFBNumwithoutFoliationCleanRe}
    \end{subfigure}
    \caption*{Numeric generation of a minimal set and the corresponding flowbox coordinates}
    \caption{Vector field (blue), level sets of the flowbox coordinates}
    \label{fig:2DLinReal}
\end{figure*}

\subsubsection{Complex eigenvalues}
\paragraph{\bf{Analytic part}} The eigenpairs of $A_C$ are  
\begin{equation*}
        \lambda_{1,2}=-\frac{9}{20}\pm \frac{\sqrt{15}}{20}i,\quad \bm{v}_{1,2}=\left(\begin{matrix}
            -\dfrac{\sqrt{5}}{20}\mp \dfrac{\sqrt{3}}{4}i\\
            2\dfrac{\sqrt{5}}{5}
        \end{matrix}\right)
\end{equation*}
and the solution is 
\begin{equation}
    \bm{x}(t) = a_1\bm{v}_1\exp\{\lambda_1 t\}+a_2\bm{v}_2\exp\{\lambda_2 t\}.
\end{equation}
The solutions of the Koopman PDE are $\Phi_1(\bm{x}) = \inp{\bm{v}_1^\perp}{\bm{x}}$ and $\Phi_2(\bm{x}) = \inp{\bm{v}_2^\perp}{\bm{x}}$ where 
\begin{equation}
\bm{v}_1^\perp=\begin{bmatrix}
    -\dfrac{2}{\sqrt{5}}i\\
\dfrac{\sqrt{3}}{4} - \dfrac{\sqrt{5}}{20}i
\end{bmatrix},\quad 
\bm{v}_2^\perp=\begin{bmatrix}
    \dfrac{2}{\sqrt{5}}i\\
\dfrac{\sqrt{3}}{4} + \dfrac{\sqrt{5}}{20}i
\end{bmatrix}.
\end{equation}
However, these solutions are not orthogonal. A minimal set yielded from these solution is
\begin{equation}
    \begin{split}
        M_1(\bm{x})&=\frac{1}{\lambda_1}\ln{\abs{\Phi_1(\bm{x})}},\quad 
        M_2(\bm{x})=\frac{1}{\lambda_2}\ln{\abs{\Phi_2(\bm{x})}}
    \end{split}.
\end{equation}
In the same vein, the flowbox coordinates are 
\begin{equation}
    \begin{split}
        z_1&=\frac{M_1+M_2}{2},\quad 
        z_2=\frac{M_1-M_2}{2}
    \end{split}
\end{equation}

\paragraph{\bf{Numeric part}} We summarize the the analytic and the numeric results in \cref{fig:2DLinComplex}.

\begin{figure*}[phtb!]
    \centering
    \captionsetup[subfigure]{justification=centering}
    \begin{subfigure}[t]{0.48\textwidth} 
\includegraphics[width=0.85\textwidth,valign = t]{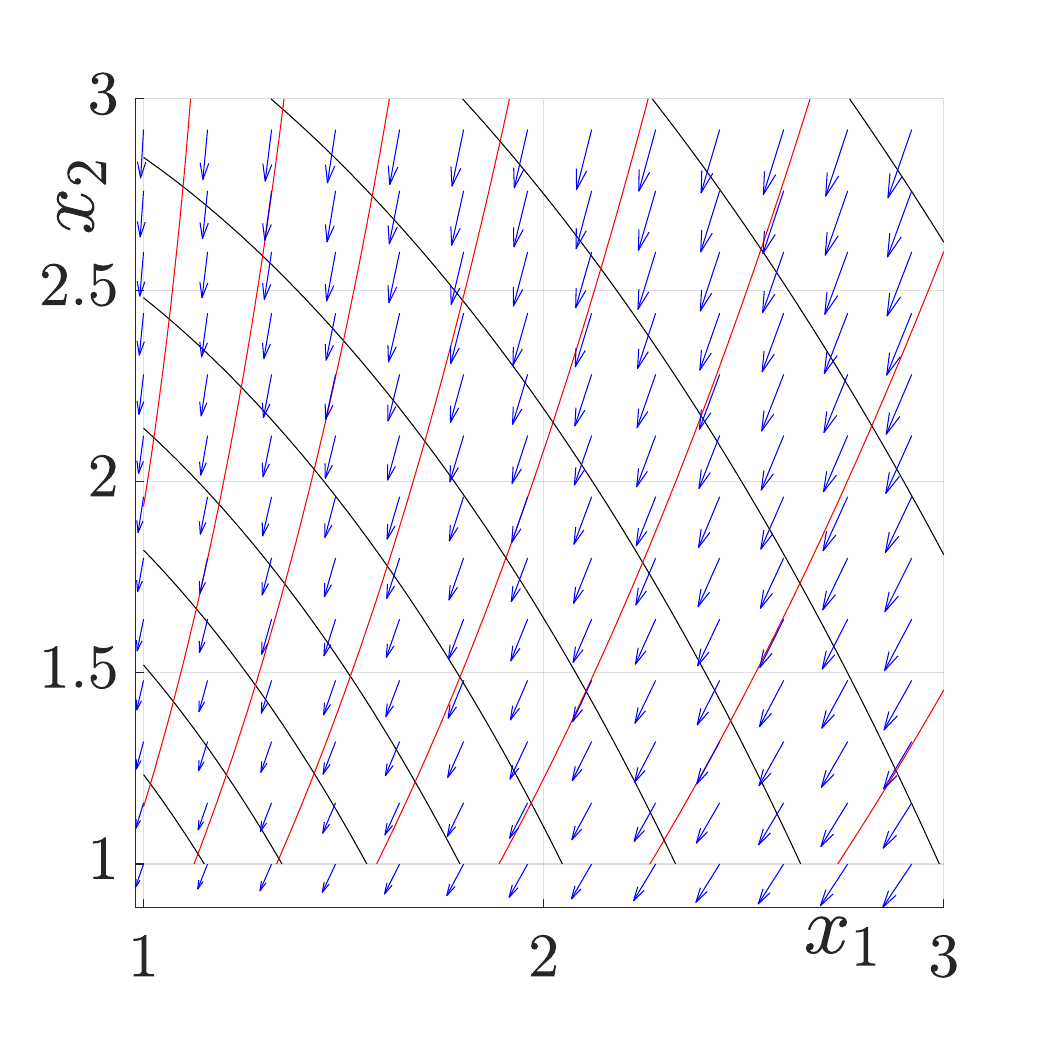}
        \label{subfig:linearSysComplexEVFBcon}
    \end{subfigure}
    \begin{subfigure}[t]{0.48\textwidth} 
\includegraphics[width=0.85\textwidth,valign = t]{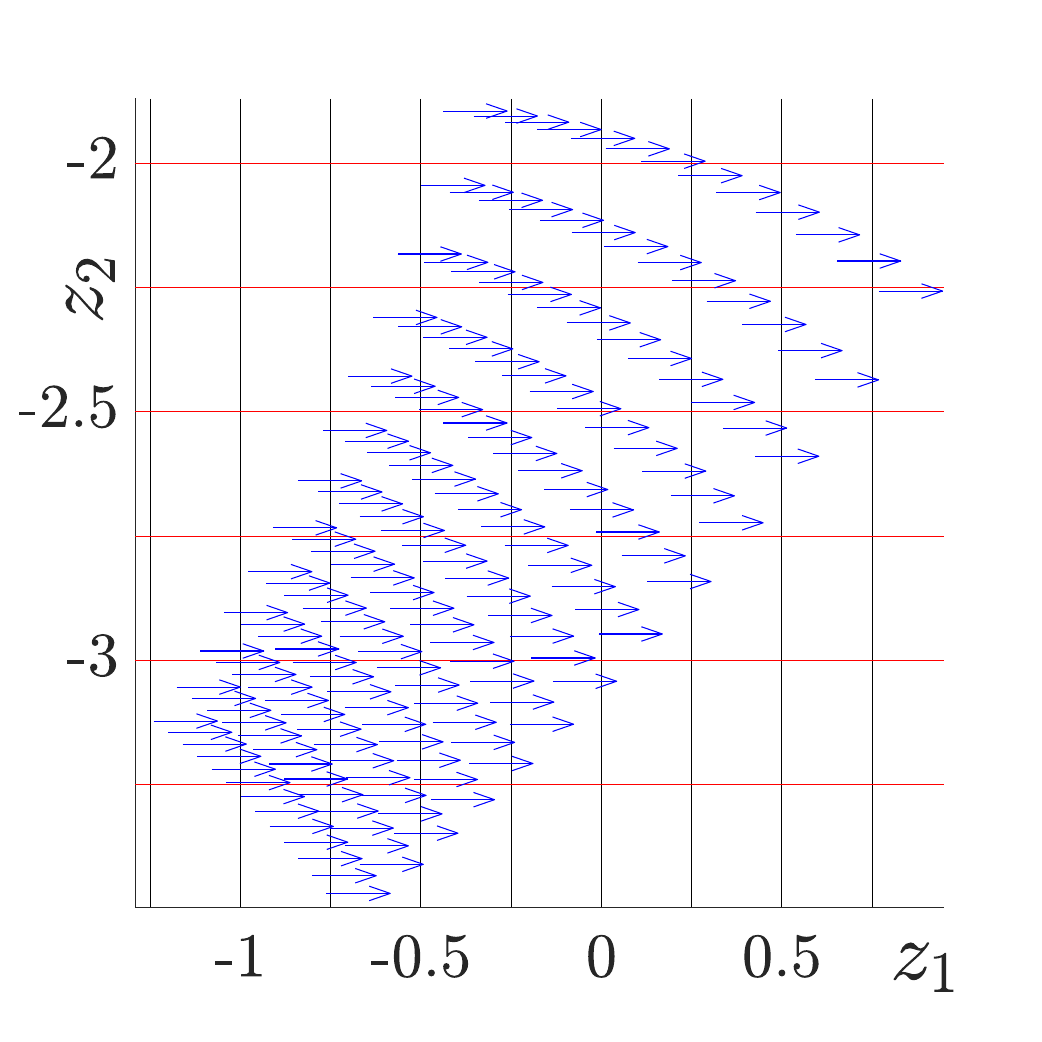}
        \label{subfig:linearSysComplexEVFB}
    \end{subfigure}\\
    \caption*{Analytic solution of the system $A_C$ finding flowbox coordinates via a minimal set}
    \begin{subfigure}[t]{0.48\textwidth} 
\includegraphics[width=0.85\textwidth,valign = t]{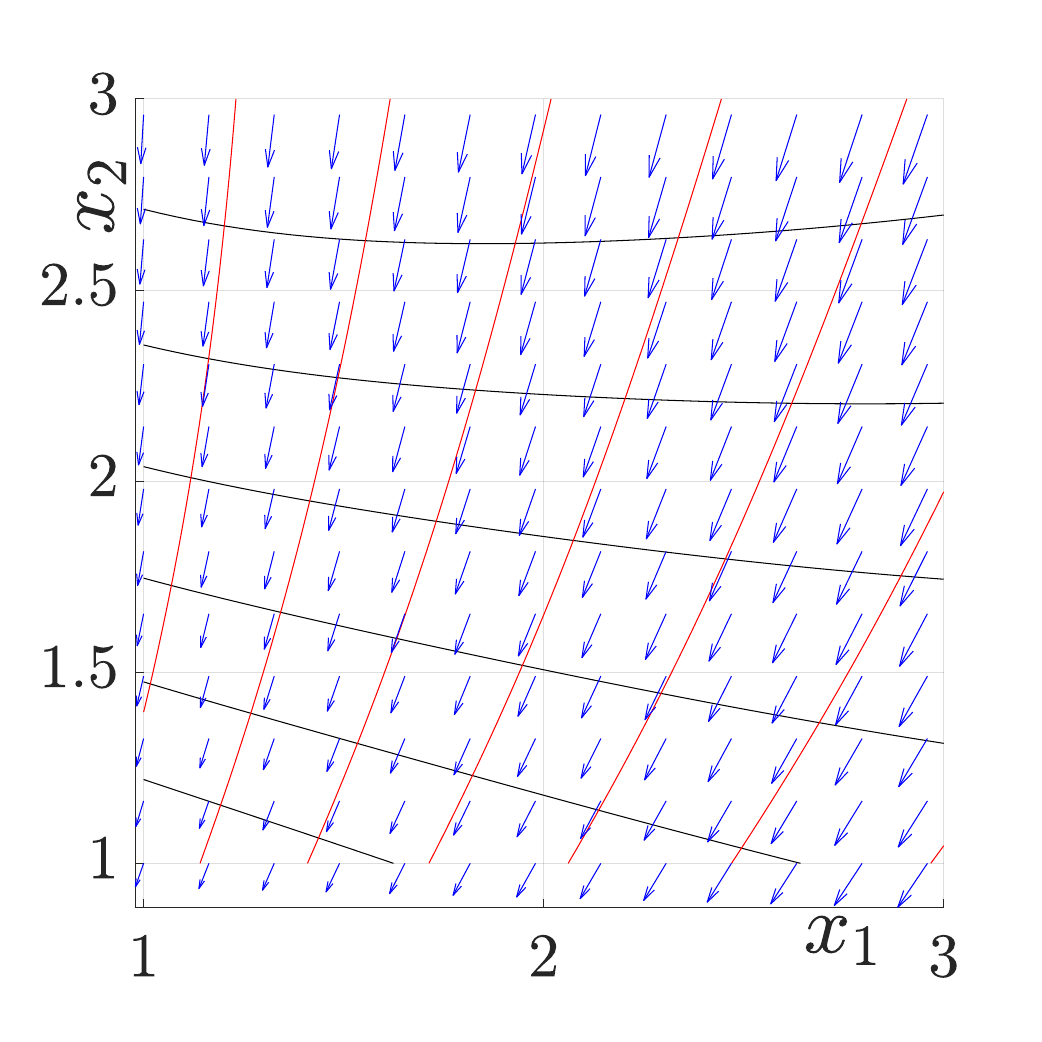}
        \label{subfig:linearSysComplexEVFBconNum}
    \end{subfigure}
    \begin{subfigure}[t]{0.48\textwidth} 
\includegraphics[width=0.85\textwidth,valign = t]{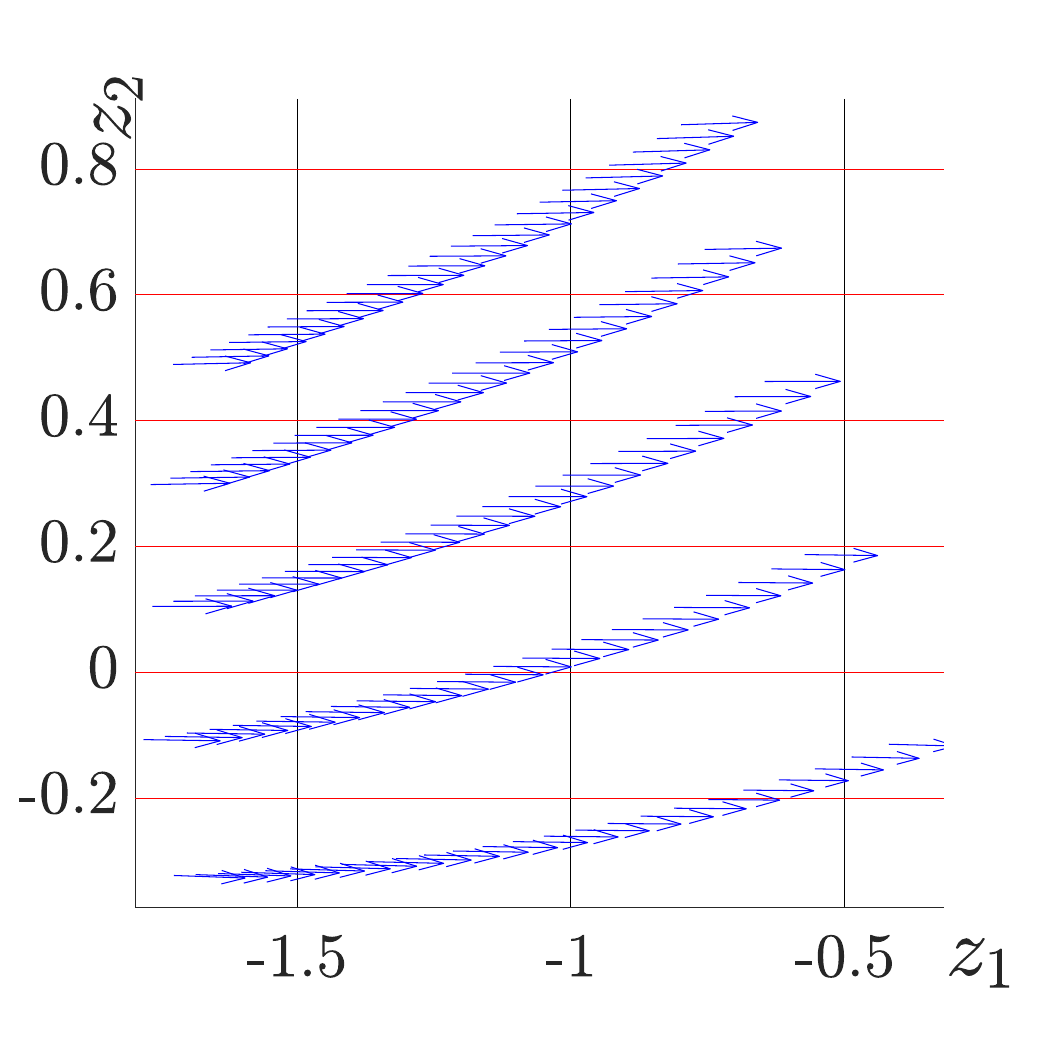}
        \label{subfig:linearSysComplexEVFBNum}
    \end{subfigure}
    \caption*{Numeric generation of a minimal set and the corresponding flowbox coordinates}
    \caption{Vector field (blue), level sets of the flowboxed coordinate system $z_1$ and $z_2$ (black and red).}
    \label{fig:2DLinComplex}
\end{figure*}

\subsubsection{Imaginary eigenvalues}
\paragraph{\bf{Analytic part}} The eigenpairs of $A_I$ are 
\begin{equation*}
        \lambda_{1,2}=\pm i,\quad \bm{v}_{1,2}=\frac{1}{\sqrt{2}}\left(\begin{matrix}
            1\\
            \pm i
        \end{matrix}\right)
\end{equation*}
and the solution is
\begin{equation}\label{eq:fullSolHarSys}
    \begin{bmatrix}
        x_1\\
        x_2
    \end{bmatrix} = \frac{a_1}{\sqrt{2}}\begin{bmatrix}
        1\\ i
    \end{bmatrix}e^{it}+\frac{a_2}{\sqrt{2}}\begin{bmatrix}
        1\\ -i
    \end{bmatrix}e^{-it}.
\end{equation}
We split the system with the linear transformation $\Phi_1(\bm{x}) = (x_1-i\cdot x_2)/\sqrt{2}$, $\Phi_2(\bm{x}) = (x_1+i\cdot x_2)/\sqrt{2}$. Canonical split dynamic coordinates are
\begin{equation}
        M_1(\bm{x})=-i \ln (\Phi_1(\bm{x})),\quad 
        M_2(\bm{x})=i \ln (\Phi_2(\bm{x}))
\end{equation}
and the flowboxed coordinates are 
\begin{equation}
z_1 = \dfrac{M_1 + M_2}{2},\quad z_2 = \dfrac{M_1 - M_2}{2}.    
\end{equation}
\paragraph{\bf{Numeric part}} In \cref{fig:2DLinImag} this system is depicted.

\begin{figure*}[phtb!]
    \centering
    \captionsetup[subfigure]{justification=centering}
    \begin{subfigure}[t]{0.48\textwidth}
\includegraphics[width=0.85\textwidth,valign = t]{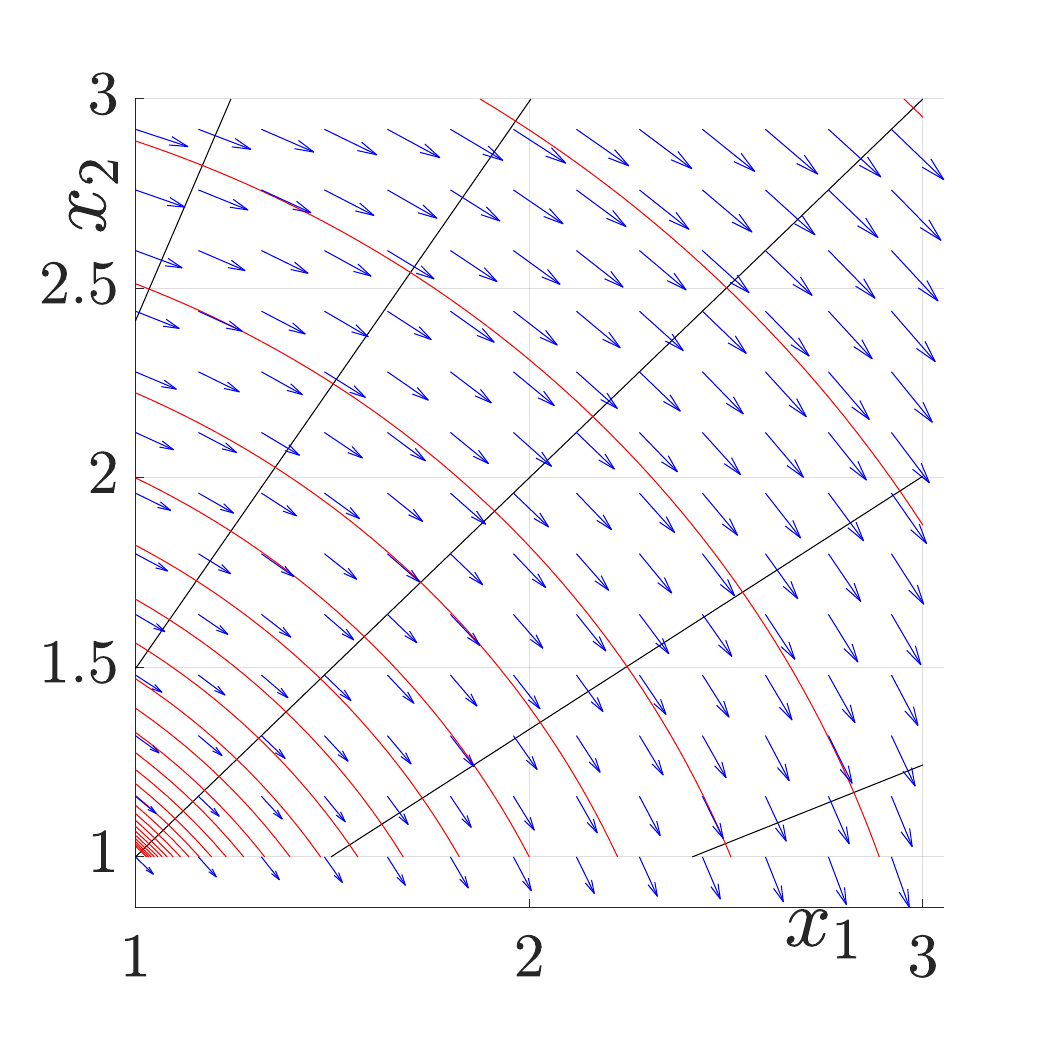}
        \label{subfig:linearSysImaginaryEVFBcon}
    \end{subfigure}
    \begin{subfigure}[t]{0.48\textwidth} 
\includegraphics[width=0.85\textwidth,valign = t]{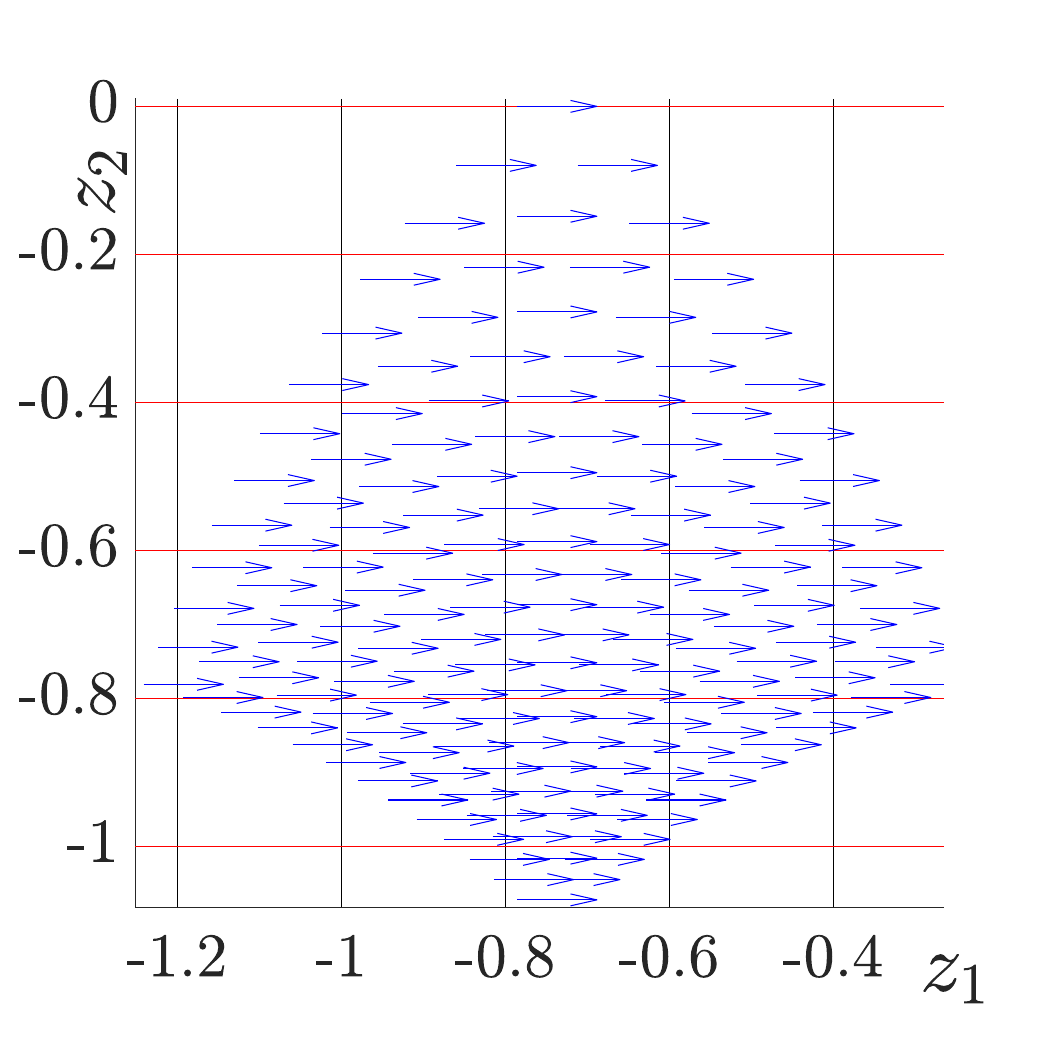}
        \label{subfig:linearSysImaginaryEVFB}
    \end{subfigure}\\
    \caption*{Analytic solution of the system, $A_I$, finding flowboxed coordinates via a minimal set}
    \captionsetup[subfigure]{justification=centering}
    \begin{subfigure}[t]{0.48\textwidth} 
\includegraphics[width=0.85\textwidth,valign = t]{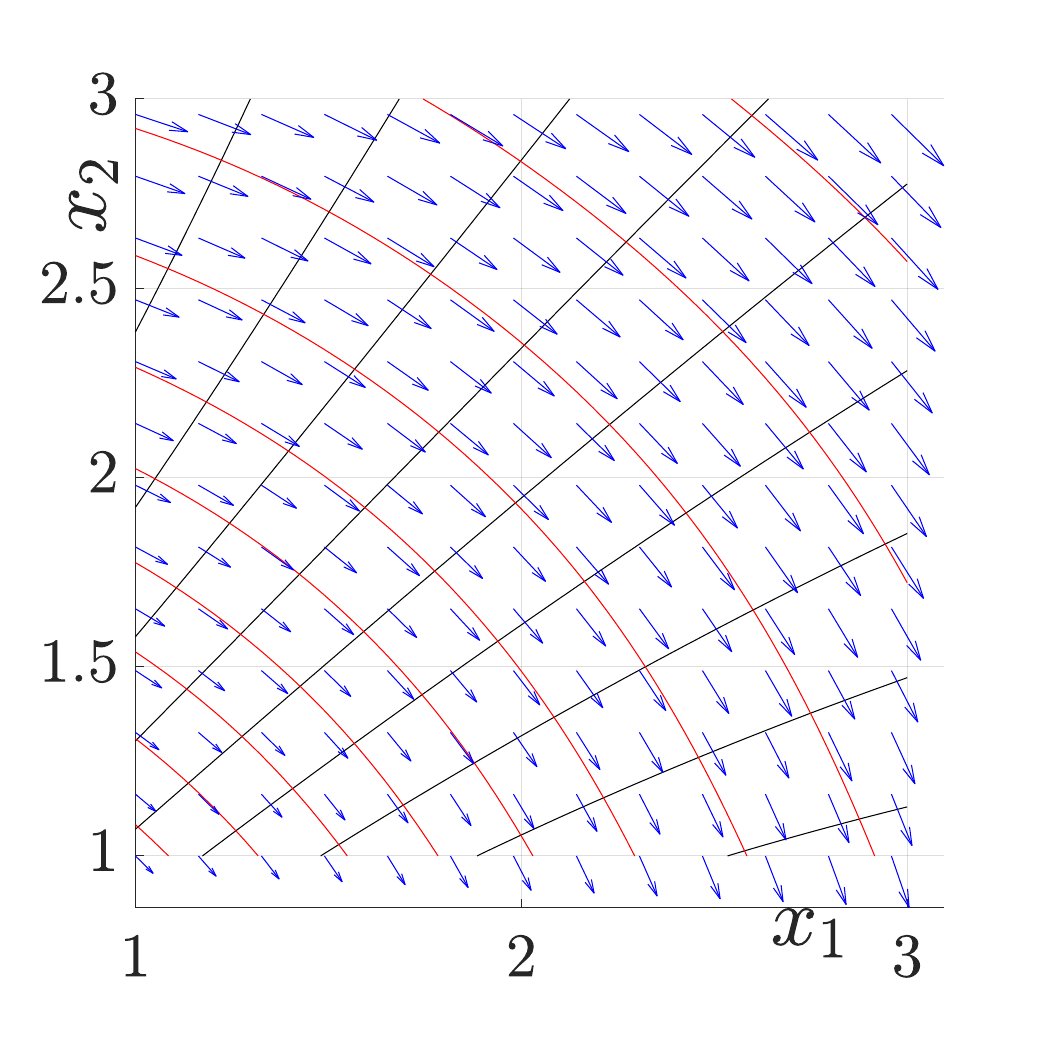}
        \label{subfig:linearSysImaginaryEVFBconNum}
    \end{subfigure}
    \begin{subfigure}[t]{0.48\textwidth} 
\includegraphics[width=0.85\textwidth,valign = t]{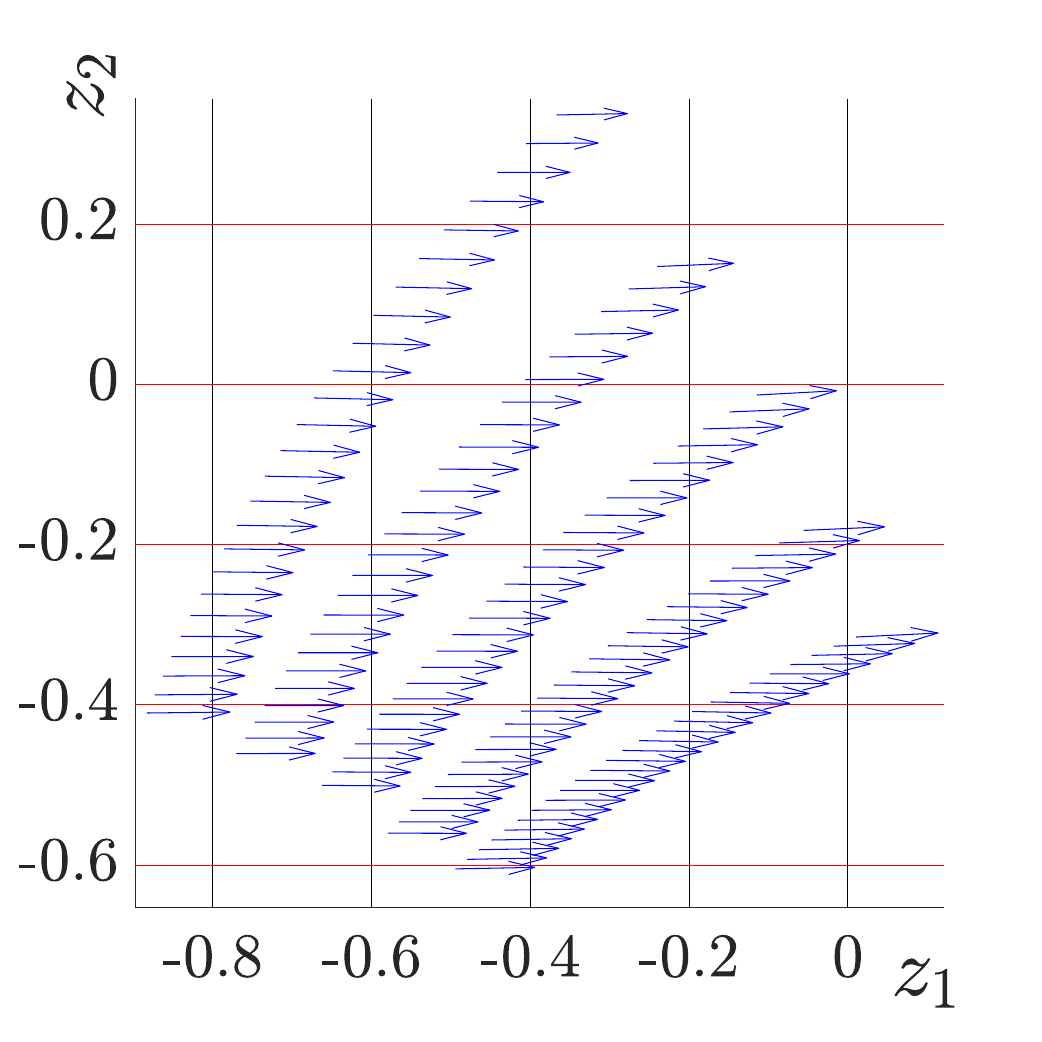}
        \label{subfig:linearSysImaginaryEVFBNum}
    \end{subfigure}
    \caption*{Numeric generation of a minimal set and the corresponding flowbox coordinates.}
    \caption{Flowboxing of a two dimensional system with imaginary eigenvalues}
    \label{fig:2DLinImag}
\end{figure*}

\subsection{Nonlinear System}
\subsubsection*{Limit cycle}
Let us consider the following dynamical system
\begin{equation}\label{eq:2DnonlinearDyn}
    \begin{split}
        \dot{x}_1 &= -x_2+ x_1 (1 - x_1^2 - x_2^2)\\
        \dot{x}_2 &= x_1+ x_2 (1 - x_1^2 - x_2^2)
    \end{split},
\end{equation}
initialized with $\bm{x}_0$ where $\bm{x}_0\ne \bm{0}, \norm{\bm{x}_0}\ne 1$. This dynamic is split with the polar coordinates, $r=\sqrt{x_1^2 +x_2^2},\, \theta = \dfrac{1}{2i}\ln\left(\dfrac{x_1+ix_2}{x_1-ix_2}\right)$.The dynamic in those coordinates are
\begin{equation}
    \begin{split}
        \dot{r}&=(1-r)(1+r)r\\
        \dot{\theta} &= 1
    \end{split}.
\end{equation}
A canonical split dynamic is
\begin{equation}
    M_1(\bm{x}) = \ln\left(\frac{\sqrt{x_1^2+x_2^2}}{\sqrt[]{\abs{1-x_1^2-x_2^2}}}\right),\quad
    M_2(\bm{x}) = \dfrac{1}{2i}\ln\left(\dfrac{x_1+ix_2}{x_1-ix_2}\right).
\end{equation}
And as usual, the flowboxed coordinates are
\begin{equation}\label{eq:2DnonlinearFB}
        z_1 = \frac{M_1+M_2}{2},\quad z_2 = \frac{M_1-M_2}{2}.
\end{equation}
\cref{fig:2DNonlin} depicts this dynamic system.

\begin{figure*}[phtb!]
    \centering
    \captionsetup[subfigure]{justification=centering}
    \begin{subfigure}[t]{0.48\textwidth} 
\includegraphics[width=0.85\textwidth,valign = t]{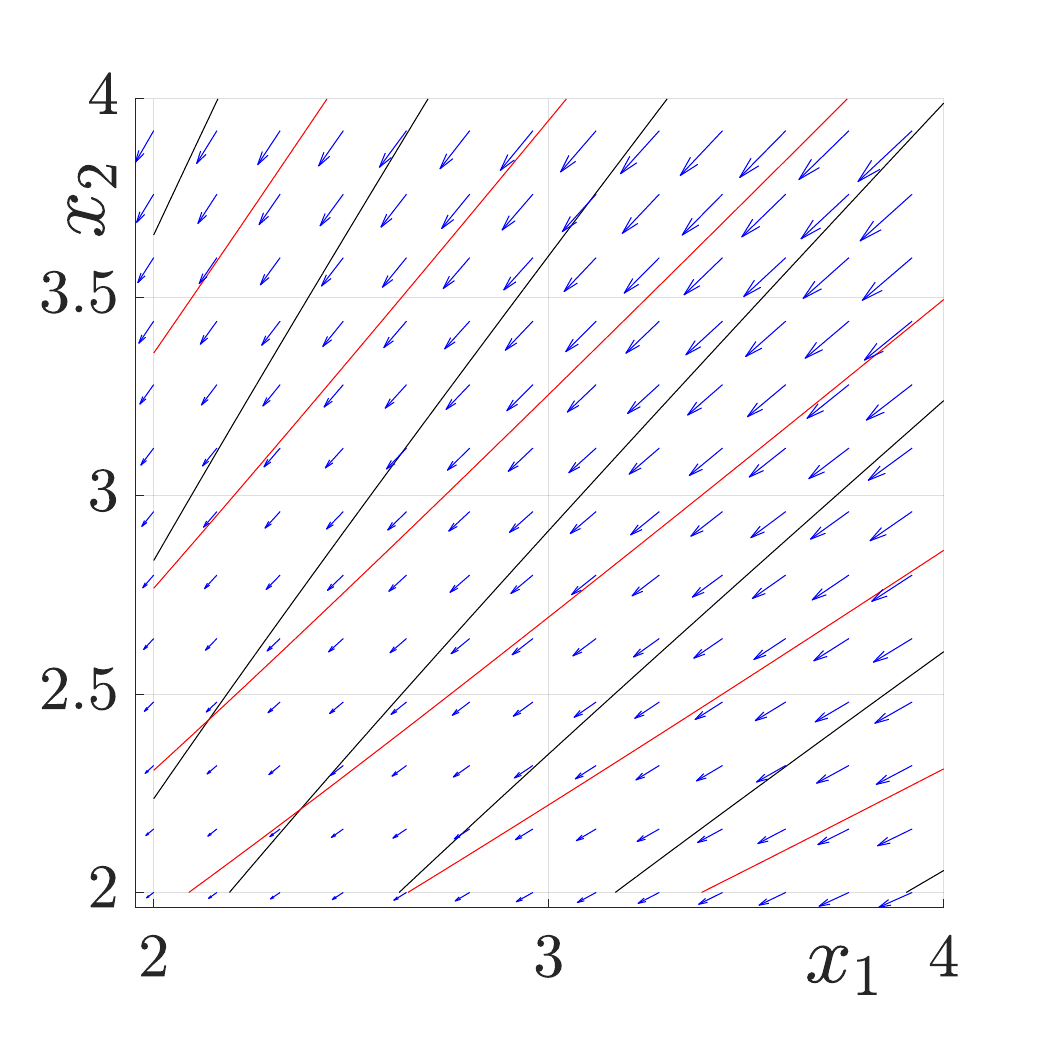}
        \label{subfig:NonlinearSysFBcon}
    \end{subfigure}
    \begin{subfigure}[t]{0.48\textwidth} 
\includegraphics[width=0.85\textwidth,valign = t]{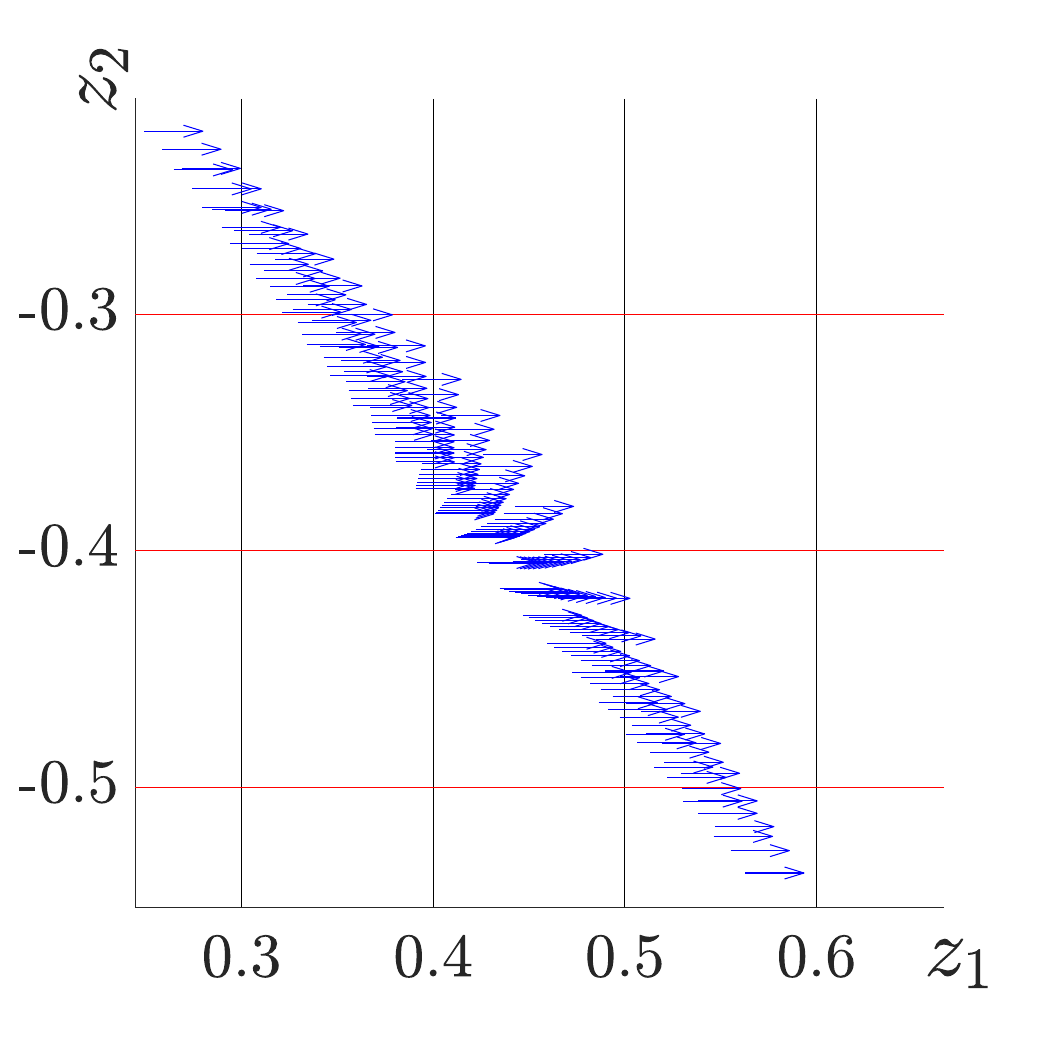}
        \label{subfig:NonlinearSysFB}
    \end{subfigure}\\
    \caption*{Analytic solution of the system, Eq. \eqref{eq:2DnonlinearDyn}, finding flowbox coordinates via a minimal set}
    \begin{subfigure}[t]{0.48\textwidth} 
\includegraphics[width=0.85\textwidth,valign = t]{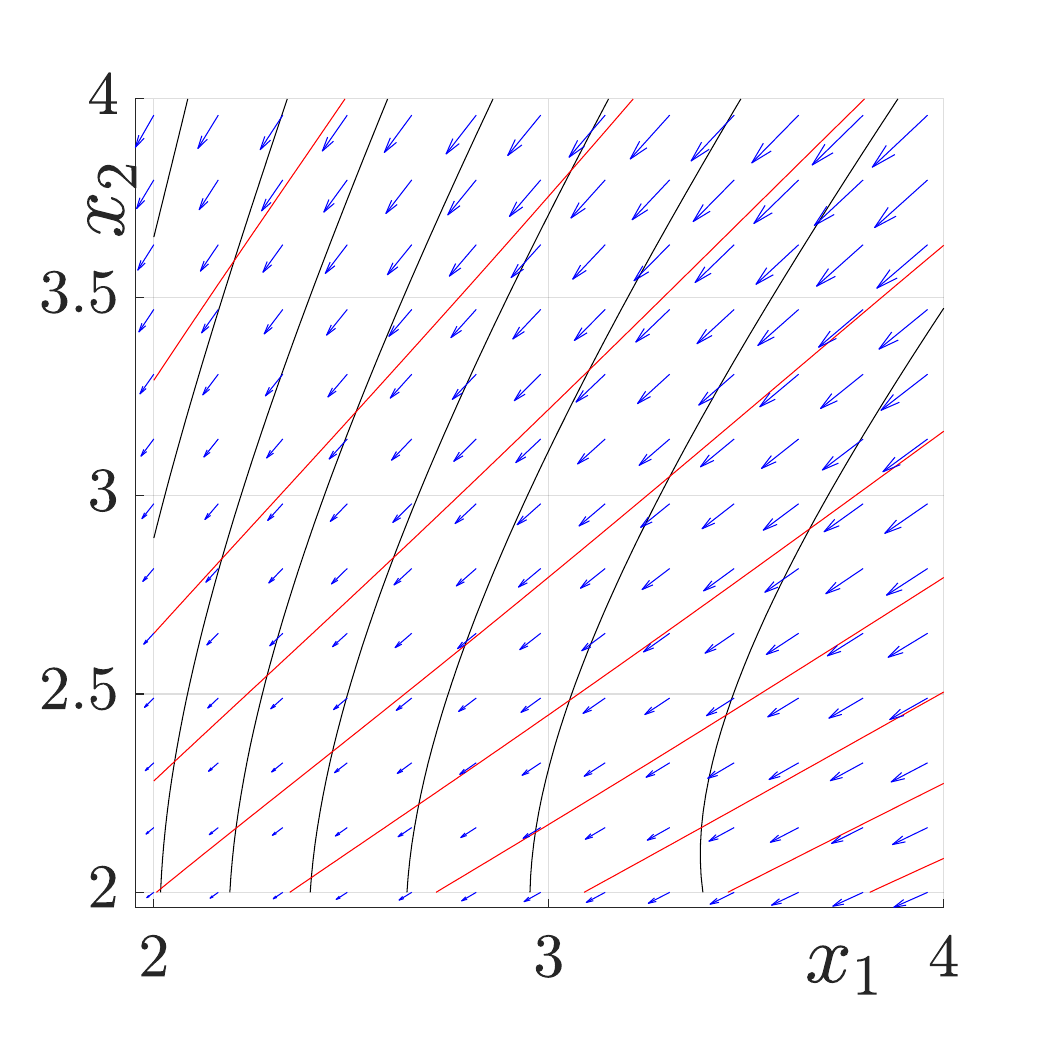}
        
        \label{subfig:NonlinSysFBconNum}
    \end{subfigure}
    \begin{subfigure}[t]{0.48\textwidth} 
\includegraphics[width=0.85\textwidth,valign = t]{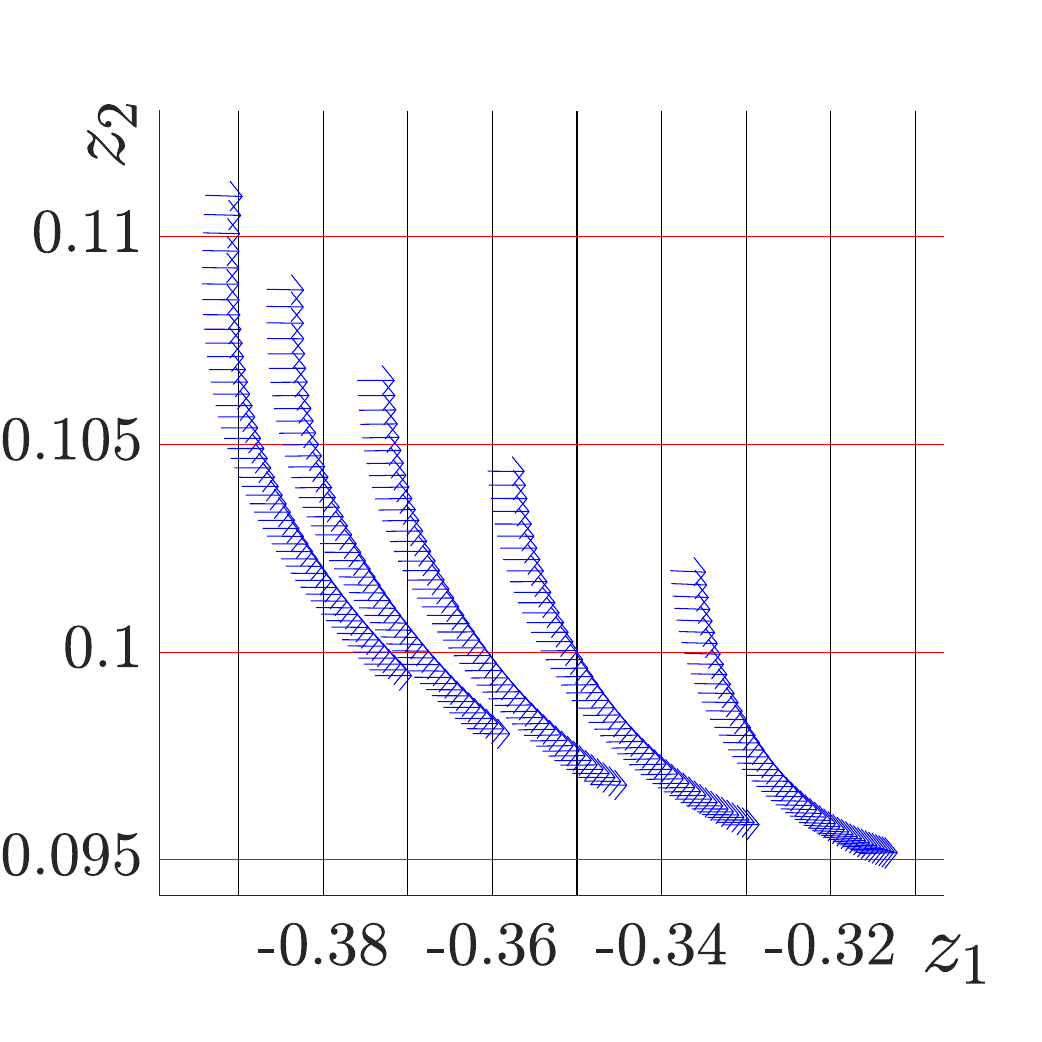}
        \label{subfig:NonlinSysFBNum}
    \end{subfigure}
    \caption*{Numeric generation of a minimal set and the corresponding flowbox coordinates}
    \caption{Minimal set and flowboxing of a nonlinear two-dimensional system}
    \label{fig:2DNonlin}
\end{figure*}

\section{Discussion and conclusions}
\subsection{Minimal set and Dimensionality reduction}
A minimal set of Koopman eigenfunctions (or unit velocity measurements) can be seen as a change of variables. This process is reversible since the Jacobian is a full-rank matrix. The minimal set is based on the unit manifolds not only at a certain dynamic's orbit but also in its neighborhood. For better understanding, let us reanalyze some of the examples above.

\paragraph{\bf{Limitations of unit velocity measurement numerics} }
The analytic solution of the system $A_R$ (real eigenvalues) starts with finding the eigenvectors and alignment the coordinate system accordingly. Then, axes rescaling lead the system to canonical representation (\cref{def:canonicalSplitDynamic}). Careful looking at that solution reveals an inherent problem when the initial condition is proportional to an eigenvector. In this case, the split dynamic deteriorates to a one-dimensional dynamic system. Sure, there are 2 different time mappings when the initial conditions are lying on an eigenvector. However, they can not be generalized to unit manifolds as it is shown in \cref{appsec:KEF_KM}. 

\begin{figure*}[phtb!]
    \centering
    \captionsetup[subfigure]{justification=centering}
    \begin{subfigure}[t]{0.48\textwidth} 
\includegraphics[width=0.85\textwidth,valign = t]{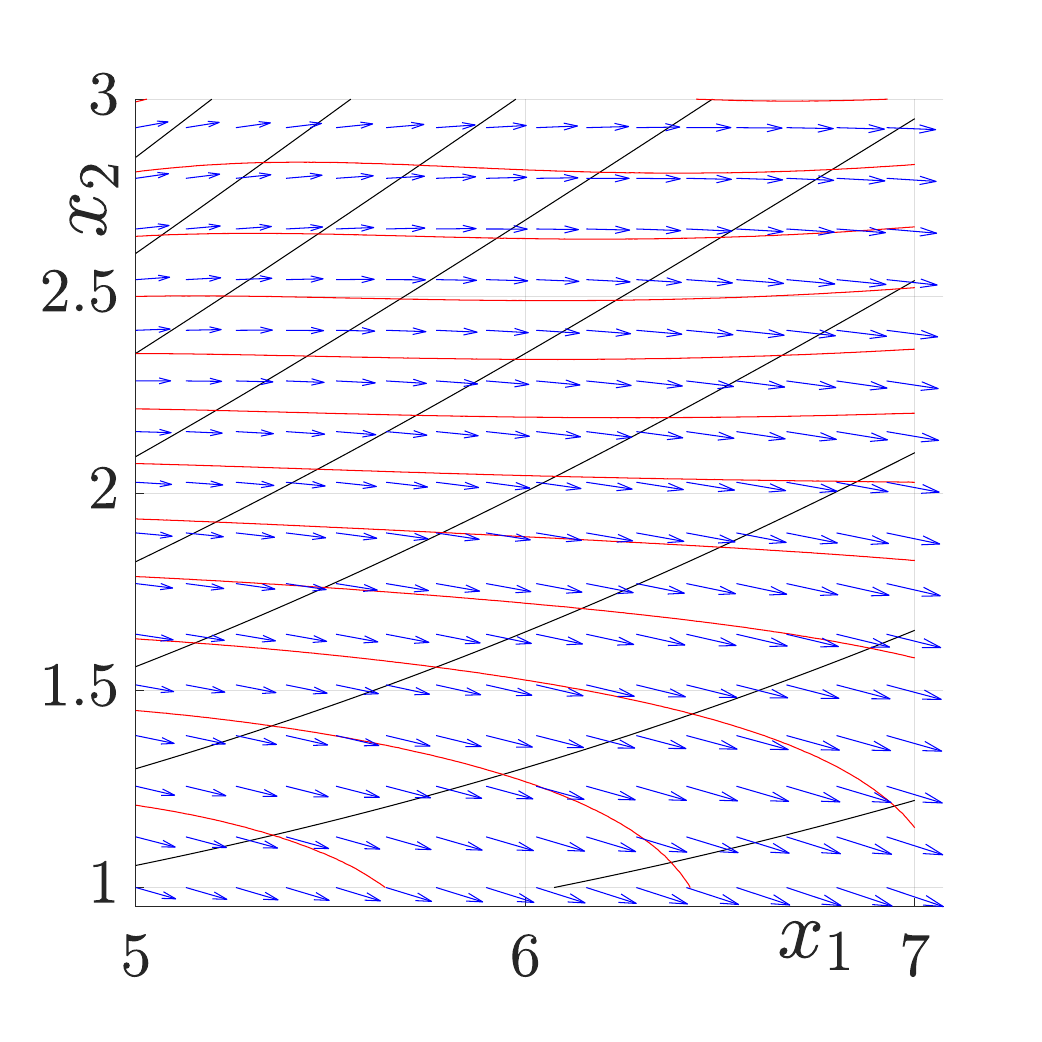}
        \label{subfig:linearSysRealEVFBconNumwithoutFoliationCleanRe2}
    \end{subfigure}
    \begin{subfigure}[t]{0.48\textwidth} 
\includegraphics[width=0.85\textwidth,valign = t]{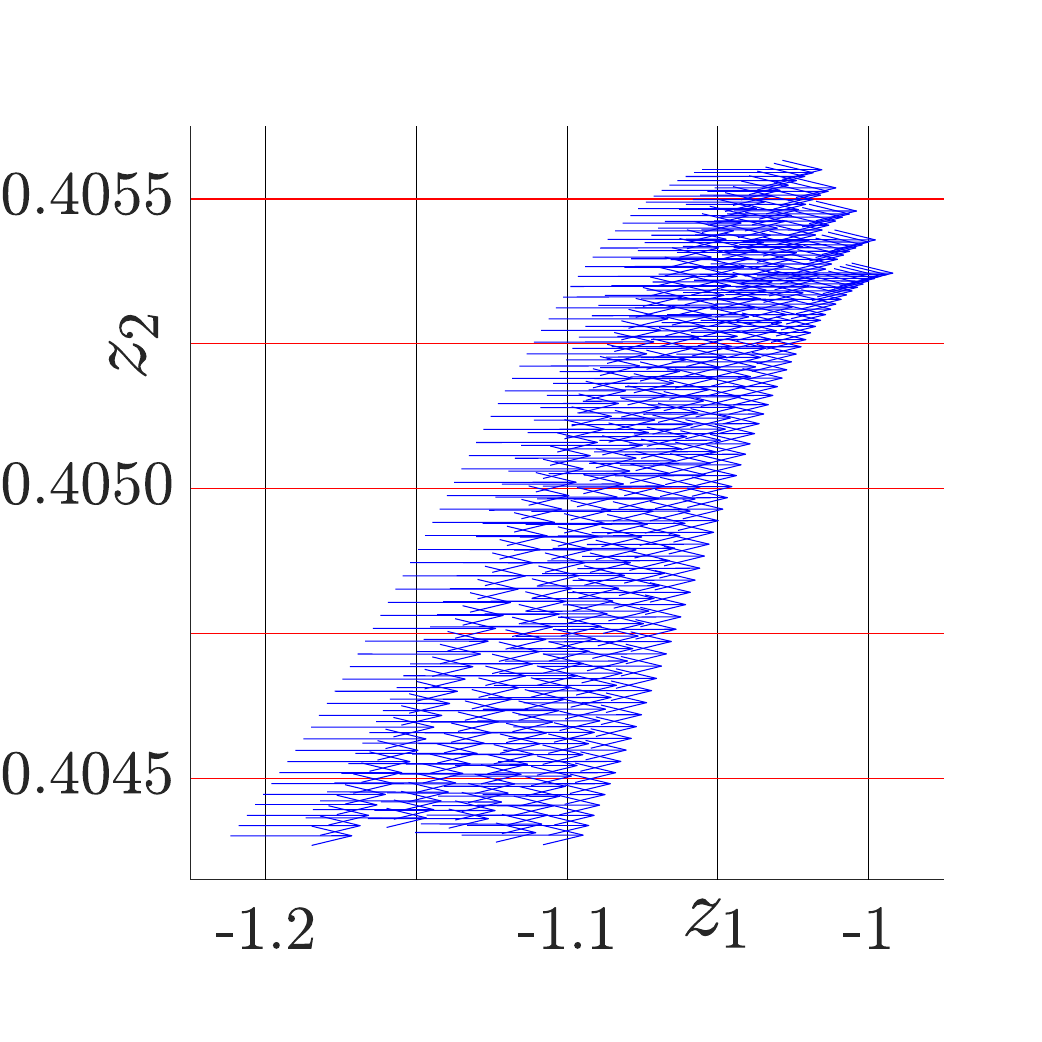}
        \label{subfig:linearSysRealEVFBNumwithoutFoliationCleanRe2}
    \end{subfigure}

    \caption{Examine the numeric solution of a linear system $A_R$ on the patches $[5,7]\times [1,3]$. }
    \label{fig:linearSysRealEVClear2} 
\end{figure*}

\begin{figure*}[phtb!]
    \centering
    \captionsetup[subfigure]{justification=centering}
    \begin{subfigure}[t]{0.48\textwidth} 
\includegraphics[width=0.85\textwidth,valign = t]{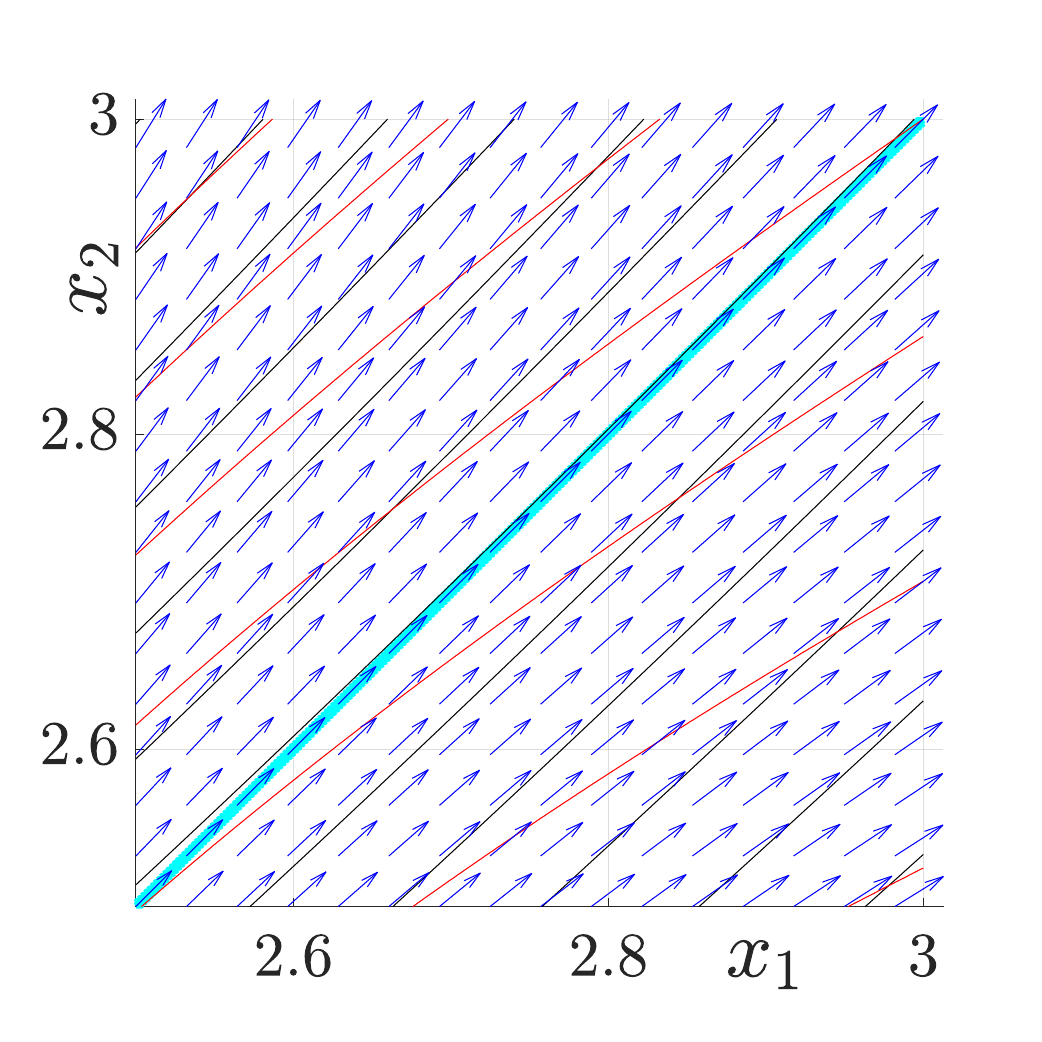}
        \label{subfig:linearSysRealEVFBconNumwithoutFoliation}
    \end{subfigure}
    \begin{subfigure}[t]{0.48\textwidth} 
\includegraphics[width=0.85\textwidth,valign = t]{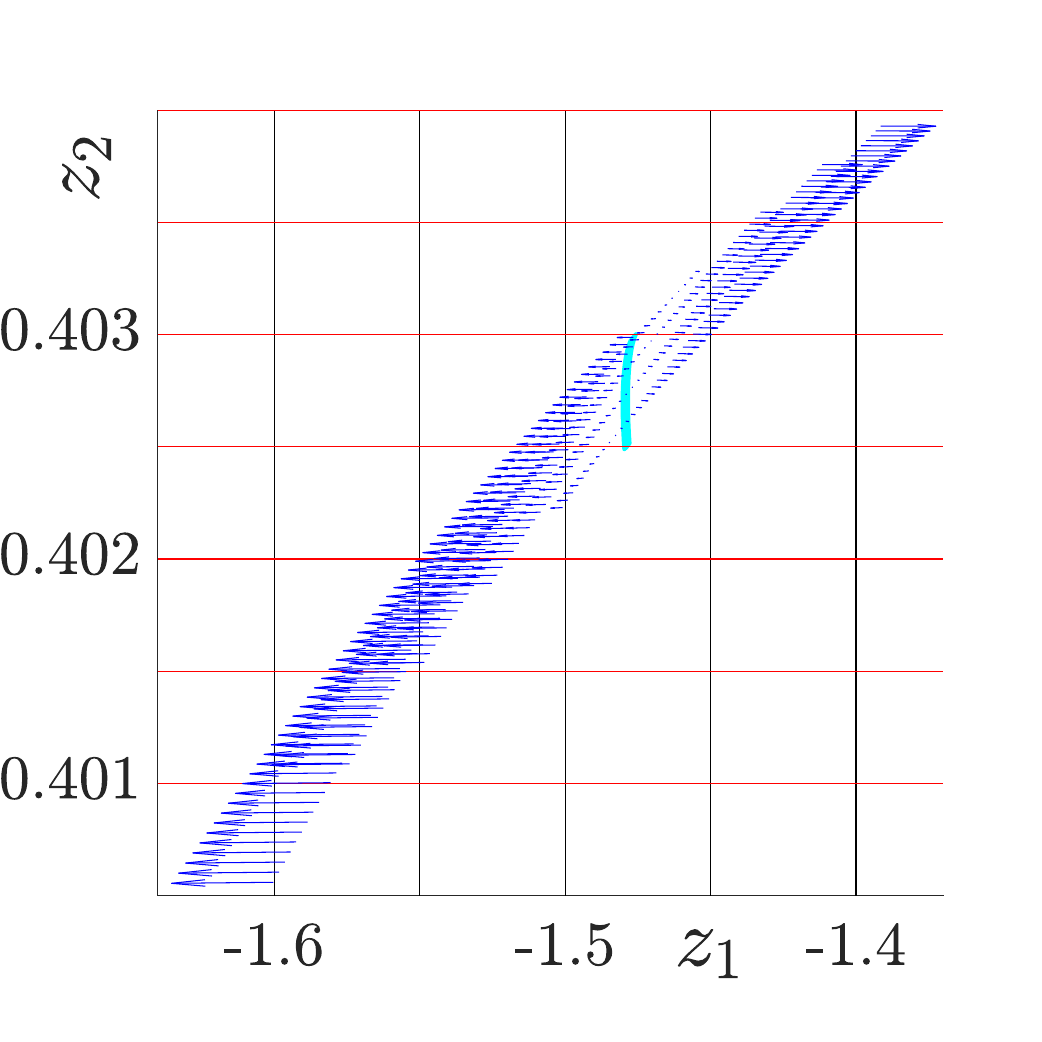}
        \label{subfig:linearSysRealEVFBNumwithoutFoliation}
    \end{subfigure}
    \caption{Examine the numeric solution of a linear system $A_R$ on the patch $[2.5,3]\times[2.5,3]$. The numeric solution where $x_1=x_2$ colored in cyan.}
    \label{fig:2DLinRealwithFoli}
\end{figure*}

In the second row in \cref{fig:2DLinReal}, the flowbox coordinates resulting from a NN are presented. This NN was trained on the patch $[4,6]\times[1,3]$. \cref{fig:linearSysRealEVClear2,fig:2DLinRealwithFoli} respectively depict the results of the NN that is fed with two different patches $[5,7]\times [1,3]$ and $[2.5,3]\times [2.5,3]$. The NN's result generalization on the patch $[5,7]\times [1,3]$ is satisfying. The variance of the error of the flowbox with respect to $z_1$ axis is $1.5594e-04$ and with respect to $z_2$ axis is $2.9356e-06$. On the other hand, as expected, the NN's result of the patch $[2.5,3]\times [2.5,3]$ is far from accurate since one of the solutions of the Koopman PDE is zero. This line and its corresponding curve in the NN result are colored in cyan.

\subsection{Linear system} 
The procedure to find the minimal set of linear systems, split dynamical system, and the flowbox coordinates as presented here can be easily generalized to higher dimensional linear systems. The steps are as follows
\begin{enumerate}
    \item {\bf{Splitting the system}} -- Find axes for which the dynamic is split. Use the eigenvectors of the dynamic to find these axes. Align the system according to these new coordinates which split the system into $N$ independent subsystems.
    \item  {\bf{Time rescaling}} -- Rescale the axes such that for each coordinates the dynamic velocity is $1$.
    \item {\bf{Flowbox}} -- Rotate the canonical split dynamic such that the dynamic velocity at each coordinate is zero but one of them which is $N$. Rescale this axis such that the velocity is one.
\end{enumerate}

\subsection{Anchor distortions}
The minimal set concept is that there is a set of reversible distortions on the coordinate system which turn the system into a linear one, by using the theory of the Koopman operator. In the results of all systems (\cref{fig:2DLinReal,fig:2DLinComplex,fig:2DLinImag,fig:2DNonlin}), one of the flowboxed coordinates is aligned with the vector field (the red curves), in the analytic solutions and in the numeric ones. These red curves are the level set of $z_2$, admitting $\dot{z}_2=0$. Naturally, $z_2$ represents the conservation law in each system. Generally, for $N$ dimensional dynamical system, there are $N-1$ different zero velocity measurements; "different" in the sense of linearly independent gradients. Thus, $N-1$ conservation laws with the first coordinate $z_1$ we get a minimal set of the Koopman eigenfunctions. These conservation laws are the anchor distortions since they have the same level sets. On the other hand, the first coordinate admitting $\dot{z}_1=1$ can induce infinite options of level sets. All of them with velocity of one.

\subsection{Dynamic Recovery from Samples}
The flowboxing of the examples above is based on the given vector fields. However, flowboxing dynamic from samples includes also dynamic recovering. In that case, we face two main problems. The first problem is entailed by the sample density and the second is related to the diversity of the initial condition. Next, we discuss these two potential problems in detail.
 
\paragraph{\bf{Finding the unit velocity measurements from samples}}
Generally, finding a canonical split system is equivalent to finding a full rank Jacobian matrix that deform the coordinate system such that the dynamic velocity is one everywhere. One of the ways to do that is by diffusion maps \cite{coifman2006diffusion} or its variants \cite{singer2008non,peterfreund2020local}. However, this method demands high-density sampled data to assure recovery of the deformation. This high density is not very common in the dynamical system and more often the dynamic is sampled very sparsely in time and in the initial conditions. Thus, the way to overcome this obstacle is to find unit velocity measurements based on time-mapping functions.

\paragraph{\bf{Diversity in initial conditions}}
Let us demonstrate dynamic recovery using vector field or samples with an example from \cite{brunton2016koopman}. Given the following nonlinear dynamical system 
\begin{equation}
    \begin{bmatrix}
        \dot{x}_1\\
        \dot{x}_2
    \end{bmatrix} = \begin{bmatrix}
        \mu x_1\\
        \lambda(x_2-x_1^2)
    \end{bmatrix}.
\end{equation}
As shown above and noted in \cite{brunton2016koopman}, there are two different solutions of the Koopman PDE,
\begin{equation}
    \begin{split}
        \Phi_1(\bm{x}) =& x_1\\
        \Phi_2(\bm{x}) =& x_2- \frac{\lambda}{\lambda-2\mu}x_1^2
    \end{split}.
\end{equation}
The suggested linearized system is given by the substitute $y_1=x_1,\, y_2 = x_2, \, y_3 = x_1^2$
\begin{equation}
    \begin{bmatrix}
        \dot{y}_1\\ \dot{y}_2\\ \dot{y}_3
    \end{bmatrix} = \begin{bmatrix}
        \mu& 0&0\\
        0 & \lambda&-\lambda\\
        0&0&2\mu
    \end{bmatrix}\begin{bmatrix}
        y_1\\ y_2 \\ y_3
    \end{bmatrix}.
\end{equation}

Now, suppose this linearized system is sampled and time mapping are approximated with NN. The next step is to approximate the Jacobian matrix to recover the dynamic (see for example system recovery in \cite{cohen_gilboa_2023}). However, the Jacobian is a $3\times 3$ matrix, and its rank is $2$ at best because $y_1$ and $y_3$ are dependent. Therefore, generating more and more measurements does not necessarily help in system recovery.

From this simple example, one can draw the following. System recovery from samples holds the possibility of dimensionality reduction since the sample can lie on a low-dimensional manifold in the problem domain. In that case, one can formulate the dynamic more concisely. The case discussed here is equivalent to flowboxing the linear system $A_I$ from samples when the samples are only from the real field or flowboxing the linear system $A_R$ when the samples lie on an eigenvector. In all these examples, the dynamic can be formulated as a lower dimensional than the original. Therefore, one can see here the immediate relation between dimensionality and the necessary richness in the samples.

\section*{List of Acronyms}
\begin{acronym}
\acro{KEF}[KEF]{\emph{Koopman Eigenfunction}}
\acro{PDE}[PDE]{\emph{Partial Differential Equation}}

\end{acronym}
\appendix
\section{Koopman Eigenfunction vs Koopman PDE's solution}\label{appsec:KEF_KM}
Let us consider the following dynamical system
\begin{equation}
    \begin{split}
        \dot{x}_1 &= x_1\\
        \dot{x}_2 & = -x_2 + x_1^2
    \end{split}.
\end{equation}
Given the initial condition $\bm{x}_0=\begin{bmatrix}
    x_{10}&x_{20}
\end{bmatrix}^T$, the solution is 
\begin{equation}
    \begin{split}
        x_1(t) &=x_{10}e^t\\
        x_2(t) &=\left(x_{20}-\frac{x_{10}^2}{3}\right)e^{-t}+\frac{x_{10}^2}{3}e^{2t}
    \end{split}.
\end{equation}
The solution when $\bm{x}_0=\begin{bmatrix}
    1&\dfrac{1}{3}
\end{bmatrix}^T$ is
\begin{equation}
    \begin{split}
        x_1(t) &=e^t\\
        x_2(t) &=\frac{1}{3}e^{2t}
    \end{split}.
\end{equation}
Obviously, $\varphi_1({\bm{x}})=x_1$ and $\varphi_2({\bm{x}})=x_2$ are \acp{KEF} where $\lambda_1=1,\, \lambda_2=2$. Even though $\varphi_2({\bm{x}})=x_2$ is continuous and smooth it does not admit Eq. \eqref{eq:KEFPDE}
\begin{equation}
    \begin{split}
        \nabla \varphi_2(\bm{x})^TP(\bm{x})&=\lambda_2 \varphi_2(\bm{x})\\
        \begin{bmatrix}
            0&1
        \end{bmatrix}\begin{bmatrix}
            x_1\\ -x_2 + x_1^2
        \end{bmatrix}&=-x_2 + x_1^2\ne 2x_2
    \end{split}
\end{equation}
unless $\bm{x}\in \mathcal{X}(1,1/3)$.

\bibliography{smartPeople}

\end{document}